\documentclass[12pt]{article}    

\usepackage{fullpage} 
\usepackage{setspace}
\usepackage{mathtools}
\usepackage{amssymb}
\usepackage{mathrsfs}
\usepackage{graphicx}
\graphicspath{{images/}}
\usepackage{caption}
\usepackage{enumerate}
\usepackage{amsthm}
\usepackage{algorithm}
\usepackage{algpseudocode}
\usepackage[dvipsnames]{xcolor}
\usepackage{titlesec}
\usepackage[utf8]{inputenc}
\usepackage[english]{babel}
\usepackage[colorlinks = true, linkcolor = blue, citecolor = blue, urlcolor = blue]{hyperref}
\usepackage{natbib}

\usepackage{indentfirst}
    
\newtheorem{theorem}{Theorem}

\newtheorem{assumption}{Assumption}

\newtheorem{proposition}{Proposition} 
\newtheorem{corollary}{Corollary}
\newtheorem{lemma}{Lemma} 
\newtheorem{definition}{Definition}
\theoremstyle{definition}
\newtheorem{example}{Example}

\renewcommand{\d}{\mathsf{d}}

\newcommand{\Var}{\text{Var}}
\newcommand{\E}{\mathbb{E}}
\newcommand{\R}{\mathbb{R}}
\newcommand{\N}{\mathcal{N}}
\renewcommand{\P}{\mathbb{P}}
\renewcommand{\L}{\mathscr{L}}

\DeclareMathOperator*{\argmin}{argmin}

\title{\bf The Generalized Oaxaca-Blinder Estimator\thanks{Email: \texttt{gbasse@stanford.edu}. We thank Avi Feller, Winston Lin, Peng Ding and the participants of the Berkeley Causal Group for helpful comments.}}
\author{Kevin Guo \\ Stanford \and Guillaume Basse \\ Stanford}
\date{\today}

\begin{document}

\maketitle


\begin{abstract}
  After performing a randomized experiment, researchers often use ordinary-least squares (OLS) regression to adjust for baseline covariates when estimating the average treatment effect.  It is widely known that the resulting confidence interval is valid even if the linear model is misspecified.  In this paper, we generalize that conclusion to covariate adjustment with nonlinear models.  We introduce an intuitive way to use any ``simple" nonlinear model to construct a covariate-adjusted confidence interval for the average treatment effect.  The confidence interval derives its validity from randomization alone, and when nonlinear models fit the data better than linear models, it is narrower than the usual interval from OLS adjustment.
  
\end{abstract}



\section{Introduction}

\subsection{Motivation}

In this paper, we study how covariates can be used to construct more precise estimates of the sample average treatment in a completely randomized experiment.  Our investigation was inspired by Lin \citep{lin2013}, who suggests performing covariate adjustment by simply fitting an OLS model with treatment-by-covariate interactions.
\begin{align}
    \label{lins_model}
    Y_i = \mu + \tau Z_i + \beta^{\top} (\mathbf{x}_i - \bar{\mathbf{x}}) + \gamma^{\top} Z_i ( \mathbf{x}_i - \bar{\mathbf{x}})
\end{align}
In the above display, $Y_i \in \R$ is an outcome variable, $Z_i \in \{ 0, 1 \}$ denotes the treatment status of unit $i$ (1 for ``treatment", 0 for ``control"), and $\mathbf{x}_i \in \R^d$ is a vector of baseline covariates.  Remarkably, Lin showed that even if the model (\ref{lins_model}) is arbitrarily misspecified, random assignment of $Z_i$ is enough to justify standard inferences based on the regression coefficient $\hat{\tau}$.  This result is summarized (informally) in Theorem \ref{lins_result}.

\begin{theorem}
\label{lins_result}
\textup{\textbf{(Lin's result, informal)}}\\
Let $n_1$ be the size of the treatment group, and let $n_0$ be the size of the control group.  If $\min(n_0, n_1) \gg d$, then we have
\begin{align}
    \sqrt{n}( \hat{\tau} - \tau) \hspace{2mm} \dot{\sim} \hspace{2mm} \mathcal{N}(0, \sigma^2)
\end{align}
where $\tau$ is the sample average treatment effect.  Moreover, $\hat{\tau}$ is at least as efficient as Neyman's \citep{neyman} unadjusted difference-of-means estimator, and the usual confidence interval for $\tau$ based on Huber-White ``robust" standard errors is valid.
\end{theorem}

Unlike earlier work by Yang \& Tsiatis \citep{yang_tsiatis_2001}, Lin's proof does not use any probabilistic assumptions other than that the treatment assignments $Z_i$ are assigned completely at random.  It applies even if the experimental units are not randomly sampled from a larger population, which is the case in most social science experiments and clinical trials \citep{abadie_etal_2020, external_validity, randomization_forgotten}.

Theorem \ref{lins_result} has since been generalized to other experimental designs \citep{fogarty_paired_experiments, li2020rerandomization, liu_stratified} and to high-dimensional linear regression \citep{bloniarz}.  It is now widely known that covariate adjustment with linear models is never ``wrong" (at least when $n \gg d$).  However, that does not mean it is always ``right."  For example, when the outcome variable is binary, nonnegative, or highly skewed, one suspects that it may be possible to further improve precision by using nonlinear models.  There have been various clever proposals for how this might be done, but none of them have all four of the appealing properties of Lin's result:
\begin{enumerate}
    \item \textit{Statistical inference}.\\
    The method produces a confidence interval with rigorous mathematical guarantees.
    
    \item \textit{Robustness to misspecification}.\\
    The method does not require any specific assumptions about the relationship between covariates and outcomes to be valid.
    
    \item \textit{Randomization-based}.\\
    The only probabilistic assumption is that treatment assignments $Z_i$ are randomly assigned.  Validity should not be compromised if the experimental units are not randomly sampled from a larger population.
    
    \item \textit{Computational simplicity}.\\
    The estimator can be computed by practitioners without extensive programming ability, using only functions that already exist in most statistical software packages.
\end{enumerate}

Many proposals come close.  Rosenbaum \citep{rosenbaum2002} suggests forming a confidence interval by inverting a Fisher randomization test based on the residuals of an arbitrary (possibly nonlinear) model;  this satisfies 1, 3 and (arguably) 4, but the validity of the confidence interval requires a constant additive treatment effect.  The leave-one-out potential outcomes method of \cite{loop} uses any regression model to construct a randomization-unbiased estimate of $\tau$, but it is not simple to implement\footnote{An \texttt{R} package exists, but implementing the method in STATA (for example) would still be challenging.} and does not come with a confidence interval.  The literature on doubly-robust methods \citep{robins1994, kang2007, cao2009, double_machine_learning} is full of proposals satisfying 1, 2, and 4, but their theoretical justifications always assume random sampling of experimental units.

The purpose of this paper is to introduce a general-purpose method for using any sufficiently ``simple" nonlinear regression model to perform covariate adjustment in a manner that satisfies 1 -- 4. Almost every widely-used parametric model is simple enough to work, and so are some nonparametric models.  As long as the chosen nonlinear models fit the data better than linear models, then our confidence intervals are narrower than the robust standard error confidence interval from OLS adjustment.

\subsection{The potential outcomes model}

In this paper, we use the Neyman-Rubin potential outcomes model of causality \citep{neyman, rubin1974}.  We consider a finite population of $n$ experimental units, indexed by the set $\mathcal{I} = \{ 1, 2, \cdots, n \}$.  Each experimental unit consists of a triple $(y_{1i}, y_{0i}, \mathbf{x}_i)$, where $\mathbf{x}_i \in \R^d$ is a vector of covariates and $y_{1i}, y_{0i}$ are potential outcomes\footnote{This tacitly assumes that there is no interference between experimental units, e.g. the treatment assignment of unit $i$ does not affect the outcome of unit $j$ if $i \neq j$.}.  The goal is to estimate the sample average treatment effect $\tau = \tfrac{1}{n} \sum_{i = 1}^n (y_{1i} - y_{0i})$.

We adopt the framework of randomization inference, which treats all of the quantities $ \{ (y_{1i}, y_{0i}, \mathbf{x}_i) \}_{i = 1}^n$ as fixed constants \citep{bloniarz, fogarty_paired_experiments, li2020rerandomization, liu_stratified}.  The only randomness is in the treatment assignments $(Z_1, \cdots, Z_n) \sim \P_{n_1, n}$, where $\P_{n_1, n}$ is the uniform distribution on length-$n$ binary vectors $v$ with $|| v ||_1 = n_1$.  The observed outcome is $Y_i = Z_i y_{1i} + (1 - Z_i) y_{0i}$, and it is a random variable.

\section{Generalizing Oaxaca-Blinder} \label{the_estimator}

\subsection{Beyond linear adjustment}

In order to motivate our procedure, we first present another way of looking at Lin's ``interactions" estimator.  Although the estimator is defined as a coefficient in a regression model, that characterization of $\hat{\tau}$ does not illuminate why it works.  For example, it is not obvious from that characterization that covariates \textit{must} be centered -- including dummy variables coding categorical features -- in order for $\hat{\tau}$ to have model-free validity.  Without centering, the ``interactions" estimator may be badly biased even in large samples.

The reason why centering is so important is that, if all covariates are centered, then fitting the model (\ref{lins_model}) is equivalent to \textit{separately} estimating two vectors of OLS coefficients: $\hat{\theta}_1$ is estimated using only data from the treatment group, and $\hat{\theta}_0$ is estimated using only data from the control group.  The estimator $\hat{\tau}$ can be recovered from these two regressions by first imputing the unobserved potential outcomes,
\begin{align}
\label{imputation}
    \hat{y}_{ti} = 
    \left\{
    \begin{array}{ll}
    y_{ti} &Z_i = t\\
    \hat{\theta}_t^{\top} \mathbf{x}_i &Z_i \neq t
    \end{array} \right.
\end{align}
and then computing $\hat{\tau} = \tfrac{1}{n} \sum_{i = 1}^n (\hat{y}_{1i} - \hat{y}_{0i})$.  Since fitted values are not affected by centering, covariates can live on their original scale in these auxiliary regressions.  In that sense, it is more natural to think of Lin's ``interactions" estimator as an \textit{imputation} estimator.  This viewpoint is also discussed by Ding \citep{ding2018} and Chapter 7 of the textbook by Imbens \& Rubin \citep{imbens_rubin}.  In econometrics, this double-imputation procedure is known as the Oaxaca-Blinder\footnote{Although the ``interactions" estimator is algebraically equivalent to the Oaxaca-Blinder estimator, \cite{lin2013} derives its asymptotic properties in a randomization-based framework -- quite distinct from the inferential framework adopted in the econometrics literature.} method \citep{kline, oaxaca, blinder}.

This perspective suggests a natural way of using an arbitrary regression model to estimate the sample average treatment effect:  simply replace $\hat{\theta}_t^{\top} \mathbf{x}_i$ in (\ref{imputation}) with $\hat{\mu}_t(\mathbf{x}_i)$, where $\hat{\mu}_t$ is estimated (using any method) on the subset of observations with $Z_i = t$, $t \in \{ 0, 1 \}$.  We will call this procedure the \textit{generalized Oaxaca-Blinder method} -- see Algorithm \ref{generalized_oaxaca_blinder}.

\begin{algorithm}[H]
\caption{The generalized Oaxaca-Blinder method}
\label{generalized_oaxaca_blinder}
\begin{algorithmic}[1]
\State \textbf{Input}. Data $\{ (\mathbf{x}_i, Y_i, Z_i) \}_{i = 1}^n$.
\State Using data from treatment group, fit a regression model $\hat{\mu}_1$ that predicts $y_{1i}$ using $\mathbf{x}_i$.
\State Use the model $\hat{\mu}_1$ to `` fill in" the unobserved values of $y_{1i}$ using (\ref{imputation1}).
\begin{align}
\label{imputation1}
    \hat{y}_{1i} = 
    \left\{
    \begin{array}{ll}
    y_{1i} &Z_i = 1\\
    \hat{\mu}_1(\mathbf{x}_i) &Z_i = 0
    \end{array} \right.
\end{align}
\State Using data from control group, fit a regression model $\hat{\mu}_0$ that predicts $y_{0i}$ using $\mathbf{x}_i$.
\State Use the model $\hat{\mu}_0$ to ``fill in" the unobserved values of $y_{0i}$.
\begin{align}
    \hat{y}_{0i} = \left\{
    \begin{array}{ll}
    \hat{\mu}_0(\mathbf{x}_i) &Z_i = 1\\
    y_{0i} &Z_i = 0
    \end{array} \right.
\end{align}
\State \textbf{Return} $\hat{\tau} := \tfrac{1}{n} \sum_{i = 1}^n (\hat{y}_{1i} - \hat{y}_{0i})$.
\end{algorithmic}
\end{algorithm}

This procedure is so simple that it has been proposed (in one form or another) many times in different communities.  Education researchers have known about this idea since Peters (1941) \citep{peters}, and it has since appeared in applied statistics \citep{belsen, hansen_bowers}, survey sampling \citep{firth_bennett, sarndal_wright1984}, epidemiology \citep{westreich_et_al}, and economics \citep{kline, oaxaca, blinder, fairlie1999, bauer2008extension}.  Its theoretical properties are studied\footnote{These works study a slight variant where only one model is estimated (using all the data), but it contains $Z$ as a covariate.} for certain choices of $\hat{\mu}_0$, $\hat{\mu}_1$ in Rosenblum \& van der Laan \citep{rosenblum_vdl} and Bartlett \citep{bartlett_2018}, but under the assumption that experimental units are sampled randomly from a hypothetical superpopulation.  Without that assumption, theoretical results have only been established in the special case where $\hat{\mu}_1$ and $\hat{\mu}_0$ are linear or logistic models \citep{lin2013, freedman_logit, hansen_bowers}.

\subsection{Prediction unbiasedness}

Although any regression model can, in principle, be plugged into Algorithm \ref{generalized_oaxaca_blinder}, not every regression model will result in an estimator that is robust to misspecification.  The key property that is required of such a model is \textit{prediction unbiasedness}.

\begin{definition}
\textup{\textbf{(Prediction unbiasedness)}}\\
For $t \in \{ 0, 1 \}$, we say that the regression model $\hat{\mu}_t$ is \underline{prediction unbiased} if (\ref{prediction_unbiasedness}) holds with probability one.
\begin{align}
    \label{prediction_unbiasedness}
    \frac{1}{n_t} \sum_{Z_i = t} \hat{\mu}_t(\mathbf{x}_i) = \frac{1}{n_t} \sum_{Z_i = t} y_{ti}
\end{align}
\end{definition}

In words, a regression model is prediction unbiased if the average prediction on the training data always exactly matches the average outcome in the training data.  In survey sampling, this condition is called ``calibration."

Many widely-used regression models are automatically prediction unbiased\footnote{Such estimators are called ``internally bias calibrated" by Firth \& Bennett \citep{firth_bennett}.}.  For example, the first-order conditions of a canonical-link GLM imply (\ref{prediction_unbiasedness}), so linear regression, logistic regression, and Poisson regression are prediction unbiased.  Given an arbitrary prediction model $\hat{\mu}_1$, it is always possible to construct a related model $\hat{\mu}_1^{\mathsf{db}}$ which is prediction unbiased by simply subtracting off the estimated bias as in (\ref{debiasing}).  
\begin{align}
    \label{debiasing}
    \hat{\mu}_1^{\mathsf{db}}(\mathbf{x}) &= \hat{\mu}_1(\mathbf{x}) - \underbrace{\frac{1}{n_1} \sum_{Z_i = 1} (\hat{\mu}_1(\mathbf{x}) - y_{1i})}_{\text{estimated bias}}
\end{align}

Another possibility is to use the fitted values $\hat{\mu}_1(\mathbf{x}_i)$ as a covariate in an OLS regression as in (\ref{post_ols}).  The first order condition characterizing $\hat{\beta}_0$ in the least-squares problem guarantees that $\hat{\mu}_1^{\mathsf{ols2}}$ will also be prediction unbiased.
\begin{align}
\label{post_ols}
\mu_1^{\mathsf{ols2}}(\mathbf{x}) &= \hat{\beta}_0 + \hat{\beta}_1 \hat{\mu}_1(\mathbf{x}),  \quad (\hat{\beta_0}, \hat{\beta}_1) = \argmin_{(\beta_0, \beta_1)} \sum_{Z_i = 1} (y_{1i} - [\beta_0 + \beta_1 \hat{\mu}_1(\mathbf{x}_i)])^2
\end{align}

In view of these simple adjustments, prediction unbiasedness does not seriously restrict the class of permissible nonlinear models.  That being said, certain desirable features of $\hat{\mu}_1$ may not be present in $\hat{\mu}_1^{\mathsf{db}}$ or $\hat{\mu}_1^{\mathsf{ols2}}$, e.g. respecting the binary nature of the outcome variable.

When the debiasing trick (\ref{debiasing}) is used, the generalized Oaxaca-Blinder estimator is algebraically equivalent to the the augmented inverse-propensity weighted (AIPW) treatment effect estimator with known (and constant) treatment propensity (see \cite{intro_to_aipw} for an overview).  In the survey sampling literature, that estimator is known as the generalized difference estimator \citep{cassel_et_al_1976, breidt2017}.  The idea of using fitted values in an OLS regression has also appeared before in the survey sampling community, under the name ``model calibration estimator" \citep{wu_sitter_2001}.  Although those connections are mathematically fruitful, we believe that the formulation we present (``estimate $\tau$ by filling in missing values with unbiased prediction models") is much more intuitive in the context of completely randomized experiments.

\subsection{Statistical inference}

Under some additional constraints on the regression models $\hat{\mu}_1$ and $\hat{\mu}_0$ (to be discussed in Section \ref{main_results}), the confidence interval (\ref{confidence_interval}) has large-sample validity.
\begin{align}
\label{confidence_interval}
\hat{\tau} \pm z_{1 - \alpha/2} \sqrt{ \frac{\widehat{\mathsf{MSE}}(1)}{n_1} + \frac{\widehat{\mathsf{MSE}}(0)}{n_0}}
\end{align}
In the above display, $\widehat{\mathsf{MSE}}(t) = \tfrac{1}{n_t - 1} \sum_{Z_i = t} [y_{ti} - \hat{\mu}_t(\mathbf{x}_i)]^2$ is an estimate of the mean-squared error of the prediction model $\hat{\mu}_t$.  The form of the confidence interval (\ref{confidence_interval}) has some intuitive appeal:  when more accurate models are used to ``fill in" the missing values, the resulting estimator $\hat{\tau}$ is more precise.  

When the regression models are constant (i.e. $\hat{\mu}_t(\mathbf{x}) \equiv \tfrac{1}{n_t} \sum_{Z_i = t} y_{ti}$), we recover the confidence interval suggested by Neyman \citep{neyman} for the difference-of-means estimator.  The robust standard error confidence interval based on Lin's ``interactions" estimator is -- in large samples -- statistically equivalent to the interval (\ref{confidence_interval}) when both regression models are linear models.  Therefore, when nonlinear regression models fit the data better than linear models, the interval (\ref{confidence_interval}) is shorter than the ``robust" standard error interval from OLS adjustment.

\subsection{An illustration} \label{fatalities_example}

Before going into theoretical details, we briefly illustrate the computational and numerical properties of the generalized Oaxaca-Blinder estimator with a simple example.  The \texttt{Fatalities} dataset in \texttt{R} was introduced by Ruhm \citep{fatalities}, and contains the number of traffic fatalities in each continental US state between 1982 and 1988 along with a few covariates.  To study the effect of a (fictional) randomized intervention designed to reduce traffic fatalities, Lin's ``interactions" estimator is a natural baseline.  However, statistical intuition suggests that Poisson models might fit better, since the outcome variable counts the occurrences of a rare event.  

Computing the generalized Oaxaca-Blinder based on Poisson regression and its associated confidence interval using Algorithm \ref{generalized_oaxaca_blinder} and formula (\ref{confidence_interval}) is not too difficult, but there are two observations that can make it even simpler.  First, for prediction unbiased regression models (like Poisson regression), the right-hand side of the identity (\ref{predictive_projective}) is often easier to work with\footnote{The right-hand side of (\ref{predictive_projective}) is reminiscent of a ``marginal effects" calculation, and of the ``g-formula''.  We have not pursued that connection because the interpretation does not make sense in the randomization model.}.
\begin{align}
    \label{predictive_projective}
    \hat{\tau} = \frac{1}{n} \sum_{i = 1}^n (\hat{y}_{1i} - \hat{y}_{0i}) = \frac{1}{n} \sum_{i = 1}^n [\hat{\mu}_1(\mathbf{x}_i) - \hat{\mu}_0(\mathbf{x}_i)]
\end{align}

Second, the interval (\ref{confidence_interval}) can be computed by simply adding $\hat{\tau}$ to the endpoints of the confidence interval from a two-sample $t$-test\footnote{This changes the normal quantile $z_{1 - \alpha/2}$ to a $t$-quantile, which we recommend.} on the residuals of $\hat{\mu}_0$ and $\hat{\mu}_1$.  Using these two computational shortcuts, the estimator and it's confidence interval can be computed in only four lines of \texttt{R} code.
\color{blue}
\begin{verbatim}
mu1 = glm(Y ~ .-Z, family=poisson, subset(data, Z==1))
mu0 = glm(Y ~ .-Z, family=poisson, subset(data, Z==0))
tau.hat = mean(predict(mu1, data, "r") - predict(mu0, data, "r"))
tau.hat + t.test(residuals(mu1, "r"), residuals(mu0, "r"))$conf.int
\end{verbatim}
\color{black}

For comparison, computing Lin's ``interactions" estimator in \texttt{R} requires roughly the same amount of code, at least when the features include at least one factor variable.  In fact, it may be easier for some users to use this method to compute $\hat{\tau}$ even when $\hat{\mu}_1$ and $\hat{\mu}_0$ are linear models, because centering factors in STATA/SAS/Excel is nonstandard.

We ran the above code 50,000 times on the \texttt{Fatalities} dataset, rerandomizing the treatment assignments $(Z_1, \cdots, Z_{336})$ in each replication.  Each time, we also computed the robust standard error confidence interval based on Lin's ``interactions" estimator.  Figure \ref{fig:fatalities} plots the randomization distributions of these two estimators.  Two features are immediately clear:  (i) both estimators have an approximately normal randomization distribution;  (ii) the Poisson regression generalized Oaxaca-Blinder estimator is much more efficient than the ``interactions" estimator.  

This efficiency gain is also reflected in the width of the associated confidence intervals:  95\% confidence intervals based on Poisson imputation were about 45\% shorter (on average) than the robust standard error confidence intervals\footnote{We used the ``HC3" version of the robust standard errors.  The ``HC0" standard errors studied in \cite{lin2013} did not have good coverage properties in this example.} based on Lin's ``interactions" estimator.  Both confidence intervals had approximately nominal coverage.  

\begin{figure}[H]
    \centering
    \includegraphics[width=12.5cm]{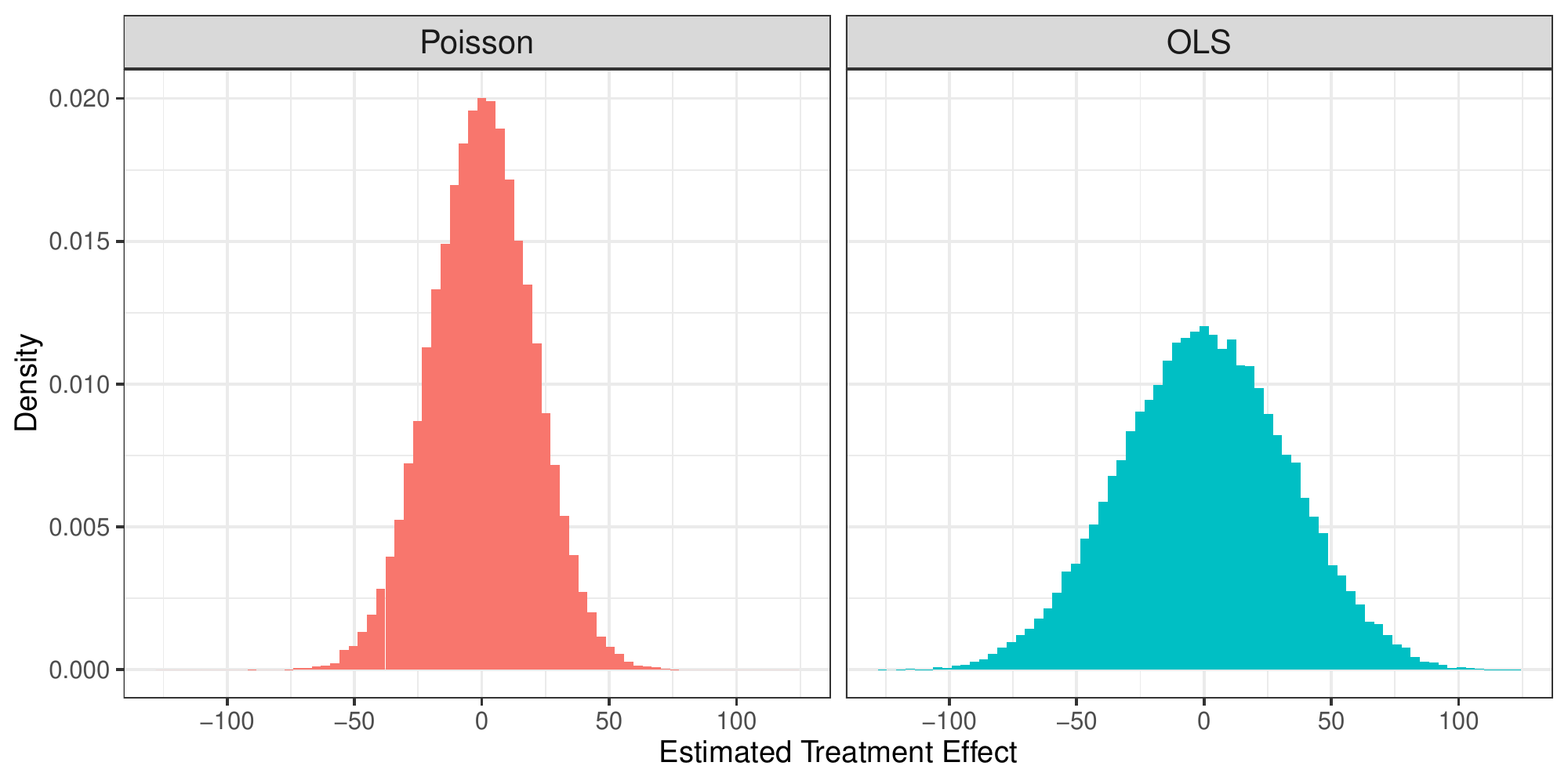}
    \caption{\textit{The randomization distribution of the Poisson regression generalized Oaxaca-Blinder estimator (left) and Lin's ``interactions" estimator (right), estimated over 50,000 randomizations.  The experimental design is completely randomized, with half of all state $\times$ year pairings receiving the ``treatment" in each randomization.  Both models control for state population, average miles per driver, and per capita income.  In the Poisson model, covariates are log-transformed.}}
    \label{fig:fatalities}
\end{figure}

\section{Theoretical results} \label{main_results}

In this section, we state our main theoretical results concerning the consistency and asymptotic normality of generalized Oaxaca-Blinder estimators.  The assumptions in this section are deliberately high-level, since the results are intended to cover a wide variety of examples.  More low-level assumptions are used to specialize these results to specific regression methods in Section \ref{examples}.

Like the results of prior work on covariate adjustment with linear models \citep{bloniarz, fogarty_paired_experiments, freedman_regression, freedman2008, li2020rerandomization, lin2013, liu_stratified}, our theoretical guarantees are asymptotic.  In Neyman's finite-population model, this means triangular-array asymptotics with respect to a sequence of finite populations $\Pi_n = \{ (y_{1i,n}, y_{0i,n}, \mathbf{x}_{i,n}) \}_{i = 1}^n$ of increasing size, each with its own completely randomized experiment $(Z_{1,n}, \cdots, Z_{n,n}) \sim \P_{n_1,n}$ and treatment effect $\tau_n = \tfrac{1}{n} \sum_{i = 1}^n (y_{1i,n} - y_{0i,n})$.  We focus on the low-dimensional regime where $d$ stays fixed as $n$ grows.  Although we do not assume that these experiments are related in any way, we will assume in what follows that the fraction of treated units $p_n = n_{1,n} / n$ satisfies $0 < p_{\min} \leq p_n \leq p_{\max} < 1$ for some bounds $p_{\min}, p_{\max}$ that do not vary with $n$.  

Some remarks on notation:  for simplicity, we will often drop the $n$-subscript on various quantities, e.g. we will write $\mathbf{x}_i$ in place of $\mathbf{x}_{i,n}$.  For two symmetric matrices $\mathbf{A}, \mathbf{B} \in \R^{d \times d}$, we will write $\mathbf{A} \succeq \mathbf{B}$ if $\mathbf{A} - \mathbf{B}$ is positive semidefinite.  For any functions $f, g : \R^d \rightarrow \R$, we define $|| f - g ||_n = (\tfrac{1}{n} \sum_{i = 1}^n [f(\mathbf{x}_{i,n}) - g(\mathbf{x}_{i,n})]^2 )^{1/2}$.

\subsection{Consistency}

Under very weak conditions, generalized Oaxaca-Blinder estimators based on prediction-unbiased regression models are consistent.  To build some intuition for why this is true, consider the Poisson regression example presented in Section \ref{fatalities_example}.  The regression models are of the form $\hat{\mu}_t(\mathbf{x}) = \exp( \hat{\theta}_t^{\top} \mathbf{x})$, where $\hat{\theta}_t$ solves (\ref{poisson}).
\begin{align}
\label{poisson}
    \hat{\theta}_t = \argmin_{\theta \in \R^d} \sum_{Z_i = t} [-y_{ti} \mathbf{x}_i^{\top} \theta + \exp( \theta^{\top} \mathbf{x}_i)]
\end{align}
Since the subset of observations with $Z_i = t$ is a random sample of all the experimental units, we would expect that $\hat{\theta}_t$ is close to the solution of the \textit{population} version of the problem (\ref{poisson}), where the sum is taken over all $i$ instead of only those with $Z_i = t$.  Let $\theta_t^*$ be the solution of the ``population" problem.  When the covariates $\mathbf{x}_i$ include an intercept, the first-order condition characterizing $\theta_t^*$ implies (\ref{poisson_matching}).
\begin{align}
\label{poisson_matching}
    \frac{1}{n} \sum_{i = 1}^n \exp(\theta_t^{* \top} \mathbf{x}_i) = \frac{1}{n} \sum_{i = 1}^n y_{ti}
\end{align}
Therefore, we could (heuristically) argue:
\begin{align*}
\frac{1}{n} \sum_{i = 1}^n \hat{y}_{ti} &= \frac{1}{n} \left( \sum_{Z_i = t} y_{ti} + \sum_{Z_i \neq t} \exp(\hat{\theta}_t^{\top} \mathbf{x}_i) \right) \\
&= \frac{1}{n} \left( \sum_{Z_i = t} \exp(\hat{\theta}_t^{\top} \mathbf{x}_i) + \sum_{Z_i \neq t} \exp( \hat{\theta}_t^{\top} \mathbf{x}_i) \right) &\text{(Prediction unbiased)}\\
&\approx \frac{1}{n} \sum_{i = 1}^n \exp(\theta_t^{* \top} \mathbf{x}_i)\\
&= \frac{1}{n} \sum_{i = 1}^n y_{ti} &\text{(By (\ref{poisson_matching}))}
\end{align*}
Since this argument works for both $t = 0$ and $t = 1$, we have $\hat{\tau} = \tfrac{1}{n} \sum_{i = 1}^n (\hat{y}_{1i} - \hat{y}_{0i}) \approx \tfrac{1}{n} \sum_{i = 1}^n (y_{1i} - y_{0i}) = \tau$. 

This simple argument is well-known in the survey sampling community (see \cite{firth_bennett, sarndal_wright1984, kang2007}), and it is the main idea behind Freedman's consistency result for logistic regression \citep{freedman2008}.  It is almost completely rigorous.  The only step that needs to be justified is the claim that $\exp(\hat{\theta}_t^{\top} \mathbf{x}_i) \approx \exp(\theta_t^{* \top} \mathbf{x})$ (at least on average).  A sufficient\footnote{This condition is not necessary for consistency, but it plays a key role in our later results on asymptotic normality.} condition to make this argument rigorous is \textit{stability}.

\begin{definition}
\textup{\textbf{(Stability)}}\\
We say that a sequence of random functions $\{ \hat{\mu}_{n} \}_{n \geq 1}$ is \underline{stable} if (\ref{stability}) holds for some deterministic sequence of functions $\{ \mu_{n}^* \}_{n \geq 1}$.
\begin{align}
    \label{stability}
    || \hat{\mu}_{n} - \mu_{n}^* ||_n := \left( \frac{1}{n} \sum_{i = 1}^n [ \hat{\mu}_{n}( \mathbf{x}_{i,n}) - \mu_{n}^* ( \mathbf{x}_{i,n})]^2 \right)^{1/2} \xrightarrow{p} 0
\end{align}
\end{definition}

The deterministic sequence in the definition of stability is not uniquely determined, but there is usually a natural choice.  For example, if $\hat{\mu}_n = \mu_{\hat{\theta}_n}$ is a parametric regression model estimated via maximum likelihood or empirical risk minimization, $\mu_n^* = \mu_{\theta_n^*}$ is the clear candidate.  For this reason, we will typically call $\mu_n^*$ ``the" population regression function, even without specifying exactly which choice of $\mu_n^*$ we are making.  

Perhaps surprisingly, the definition of stability does not assume that the deterministic sequence $\{ \mu_n^* \}_{n \geq 1}$ satisfies a property like (\ref{poisson_matching}).  It turns out that if $\{ \hat{\mu}_{1,n} \}_{n \geq 1}$ is both prediction unbiased and stable, then we may always choose the sequence $\{ \mu_{1,n}^* \}_{n \geq 1}$ to satisfy $\tfrac{1}{n} \sum_{i = 1}^n \mu_{1,n}^*(\mathbf{x}_i) = \tfrac{1}{n} \sum_{i = 1}^n y_{1i}$.  Therefore, prediction unbiasedness and stability are sufficient to carry through the heuristic argument from above -- Theorem \ref{consistency} gives a formal statement.

\begin{theorem}
\label{consistency}
\textup{\textbf{(Consistency)}}\\
Let $\{ \hat{\mu}_{1,n} \}_{n \geq 1}$ and $\{ \hat{\mu}_{0,n} \}_{n \geq 1}$ be two stable sequences of prediction-unbiased models, and let $\{ \hat{\tau}_n \}_{n \geq 1}$ be the resulting sequence of generalized Oaxaca-Blinder estimators.  For $t \in \{ 0, 1 \}$, let $\mathsf{MSE}_n(t)$ be defined by (\ref{mse}).
\begin{align}
\label{mse}
    \mathsf{MSE}_{n}(t) := \frac{1}{n} \sum_{i = 1}^n [ \mu_{t,n}^*( \mathbf{x}_{i,n}) - y_{ti,n}]^2
\end{align}
If $\mathsf{MSE}_n(1) = o(n)$ and $\mathsf{MSE}_n(0) = o(n)$, then $(\hat{\tau}_n - \tau_n) \xrightarrow{p} 0$.
\end{theorem}

Note that Theorem \ref{consistency} does not depend on the specific choice of the sequences $\{ \mu_{1,n}^* \}_{n \geq 1}$ and $\{ \mu_{0,n}^* \}_{n \geq 1}$.  As long as any nonrandom sequence $\{ \mu_{t,n}^* \}_{n \geq 1}$ satisfying $|| \hat{\mu}_{t,n} - \mu_{t,n}^* ||_n \rightarrow_p 0$ has $\mathsf{MSE}_n(t) = o(n)$, then all such sequences will have that property.  In most cases, the mean-squared error of even a grossly misspecified model is not diverging at all, so we would have $\mathsf{MSE}_n(1) = \mathcal{O}(1)$, $\mathsf{MSE}_n(0) = \mathcal{O}(1)$.  We have used the weaker assumption in Theorem \ref{consistency} only for the sake of generality.

Proving that a sequence of random functions $\{ \hat{\mu}_n \}_{n \geq 1}$ is stable is typically an exercise in translating some standard arguments from the theory of M-estimation into the language of finite populations.  This may or may not be simple, depending on the regression function.  In Section \ref{examples}, we give a few examples of widely-used regression methods where this can be done, including logistic regression, Poisson regression, and OLS regression with a transformed outcome variable.  In Section \ref{parametric_recipe}, we outline a general strategy that works for a large class of smooth parametric models.  The argument is especially simple in the case of linear regression, so we present it as an example.

\begin{example}
\label{ols_is_stable}
\textup{\textbf{(OLS is stable)}}\\
Let $\hat{\mu}_{1,n}(\mathbf{x}) = \hat{\beta}_{1,n}^{\top} \mathbf{x}$, where $\hat{\beta}_1 = \argmin \sum_{Z_i = 1} (y_{1i,n} - \mathbf{x}_{i,n}^{\top} \beta)^2$.  Assume that $\tfrac{1}{n} \sum_{i = 1}^n || \mathbf{x}_i ||^4 = o(n)$, $\tfrac{1}{n} \sum_{i = 1}^n y_{1i}^4 = o(n)$, and $\tfrac{1}{n} \sum_{i = 1}^n (y_{1i}, \mathbf{x}_i)(y_{1i}, \mathbf{x}_i)^{\top}$ converges to an invertible matrix.  Then $\{ \hat{\mu}_{1,n} \}_{n \geq 1}$ is a stable sequence.
\end{example}

\begin{proof}
We can write $\hat{\beta}_{1,n} = ( \tfrac{1}{n_1} \sum_{Z_i = 1} \mathbf{x}_i \mathbf{x}_i^{\top})^{-1} (\tfrac{1}{n_1} \sum_{Z_i = 1} \mathbf{x}_i^{\top} y_{1i})$.  Let $\mathbf{\Sigma}_{xx}$ be the limit of $\tfrac{1}{n} \sum_{i = 1}^n \mathbf{x}_i \mathbf{x}_i^{\top}$ and $\mathbf{\Sigma}_{xy}$ be the limit of $\tfrac{1}{n} \sum_{i = 1}^n \mathbf{x}_i y_{1i}$.  By the completely randomized law of large numbers\footnote{See Lemma \ref{wlln} in the appendix.} and the continuous mapping theorem, $\hat{\beta}_{1,n} \xrightarrow{p} \beta_1^* = \mathbf{\Sigma}_{xx}^{-1} \mathbf{\Sigma}_{xy}$.  If we set $\mu_{1,n}^*(\mathbf{x}) \equiv \beta_1^{* \top} \mathbf{x}$, we may write:
\begin{align*}
|| \hat{\mu}_{1,n} - \mu_{1,n}^* ||_n^2 &= \frac{1}{n} \sum_{i = 1}^n [(\hat{\beta}_{1,n} - \beta_1^*)^{\top} \mathbf{x}_i]^2 \leq \text{Tr}(\mathbf{\Sigma}_n) || \hat{\beta}_{1,n} - \beta_1^* ||^2 \xrightarrow{p} 0
\end{align*}
Thus, $\{ \hat{\mu}_{1,n} \}_{n \geq 1}$ is stable.
\end{proof}

\subsection{Asymptotic normality}

In order prove the asymptotic normality of generalized Oaxaca-Blinder estimators, we require that the regression functions $\hat{\mu}_1$ and $\hat{\mu}_0$ satisfy one additional property, which we call \textit{typically simple realizations}.  The role of this assumption is to give more precise control on the errors that are incurred in the approximation $\tfrac{1}{n} \sum_{i = 1}^n \hat{\mu}_1(\mathbf{x}_i) \approx \tfrac{1}{n} \sum_{i = 1}^n \mu_1^*(\mathbf{x}_i)$.

\begin{definition}
\label{typically_simple_realizations}
\textup{\textbf{(Typically simple realizations)}}\\
We say that a sequence of random functions $\{ \hat{\mu}_n \}_{n \geq 1}$ has \underline{typically simple realizations} there exists a sequence of function classes $\{ \mathcal{F}_n \}_{n \geq 1}$ such that $\P_{n_1,n}( \hat{\mu}_n \in \mathcal{F}_n) \rightarrow 1$ and (\ref{entropy}) holds.
\begin{align}
\label{entropy}
    \int_0^{1} \sup_{n \geq 1} \sqrt{\log \mathsf{N}(\mathcal{F}_n, || \cdot ||_n,  s)} \, \mathsf{d} s < \infty
\end{align}
In the above display, $\mathsf{N}(\mathcal{F}_n, || \cdot ||_n, s)$ denotes the $s$-covering number\footnote{The $s$-covering number of a metric space $(T, d)$ is the size of the smallest collection of points $\{ t_1, \cdots, t_N \}$ with the property that every point in $T$ is within distance $s$ of one of the $t_i$'s. } of the metric space $(\mathcal{F}_n, || \cdot ||_n)$.
\end{definition}

The integral in (\ref{entropy}) measures the ``complexity" of the possible realizations of $\hat{\mu}_n$.  In plain language, $\hat{\mu}$ has typically simple realizations if it usually falls in a set with small complexity.  The integral (\ref{entropy}) is essentially the uniform entropy integral from classical empirical process theory, but the function class $\mathcal{F}_n$ is allowed to change with $n$ and the supremum is only over distributions of the form $\tfrac{1}{n} \sum_{i = 1}^n \delta_{\mathbf{x}_{i,n}}$.  As a result, the conditions needed for ``typically simple realizations" are a bit weaker than the conditions used to prove the triangular-array Donsker property (see Theorems 2.8.9 in \cite{vdv_wellner} for details).  For example, marginal asymptotic normality (which requires $n_1 / n \rightarrow p$ for some limit $p$) is not needed.

Although Definition \ref{typically_simple_realizations} is a somewhat technical definition, it can be fairly easy to check.  For illustration, we show how it can be established in the case of OLS regression.

\begin{example}
\label{ols_is_simple}
\textbf{(OLS has typically simple realizations)}.\\
Let $\hat{\mu}_{1,n}(\mathbf{x}) = \hat{\beta}_{1,n}^{\top} \mathbf{x}$, and assume that the conditions in Example \ref{ols_is_stable} are satisfied.  Then $\{ \hat{\mu}_{1,n} \}_{n \geq 1}$ has typically simple realizations.
\end{example}

\begin{proof}
Since $\hat{\beta}_{1,n} \xrightarrow{p} \beta_1^*$, the function $\hat{\mu}_{1,n}$ typically takes values in the ``simple" set $\mathcal{F}_n = \{ \mu_{\beta}(\mathbf{x}) := \beta^{\top}\mathbf{x} \, : \, || \beta - \beta_1^* || \leq 1 \}$.  To see that this set is simple, use the fact for all $\beta, \gamma$, $|| \mu_{\beta} - \mu_{\gamma} ||_n \leq M || \beta - \gamma ||$ for some $M < \infty$.
\begin{align*}
|| \mu_{\beta} - \mu_{\gamma} ||_n^2 &= \frac{1}{n} \sum_{i = 1}^n |(\beta - \gamma)^{\top} \mathbf{x}_i|^2 \leq \text{Tr}(\mathbf{\Sigma}_n) || \beta - \gamma ||^2 \leq M^2 || \beta - \gamma ||^2
\end{align*}
The last inequality is valid with $M^2 = 2 \text{Tr}(\mathbf{\Sigma})$ for large enough $n$, since $\mathbf{\Sigma}_n \rightarrow \mathbf{\Sigma}$.  As a consequence, the $s$-covering number of $\mathcal{F}_n$ can be bounded using the $(s/M)$-covering number of the Euclidean ball $\mathbb{B}_1(\beta_1^*) := \{ \beta \in \R^d \, : \, || \beta - \beta_1^* || \leq 1 \}$.  A simple volume argument\footnote{Let $\beta_1$ be any point in the ball, $\beta_2$ any point not within distance $s$ of $\beta_1$, $\beta_3$ any point not within $s$ of either $\beta_1$ or $\beta_2$, and so on.  The process has to terminate within $(1 + 2/s)^d$ steps or else the total volume of balls around the previously chosen $\beta_i$ will exceed the total volume of $\mathbb{B}_1(\beta_1^*)$.} shows that the $s$-covering number of a $\mathbb{B}_1(\beta_1^*)$ is less than or equal to $(1 + 2/s)^d$.  Thus, we have:
\begin{align*}
    \int_0^{1} \sup_n \sqrt{\log \mathsf{N}(\mathcal{F}_n, || \cdot ||_n, s)} \, \d s &\leq \int_0^{1} \sqrt{d \log(1 + 2M / s)} \, \d s \leq 3 d M < \infty
\end{align*}
\end{proof}

The argument above works whenever $\hat{\theta} \rightarrow \theta^*$ for some limit $\theta^*$ and $\theta \mapsto \mu_{\theta}$ is smooth near $\theta^*$.  More general nonparametric function classes can also be shown to have the ``typically simple realizations" property, using combinatorial arguments.  For example, if $\mathcal{F}_n$ is any bounded Vapnik-Chervonenkis class, then (\ref{entropy}) is satisfied.

If stable prediction-unbiased models with typically simple realizations are used in the imputation step of the generalized Oaxaca-Blinder method, then the imputed means $\tfrac{1}{n} \sum_{i = 1}^n \hat{y}_{1i}$ and $\tfrac{1}{n} \sum_{i = 1}^n \hat{y}_{0i}$ both have asymptotically linear expansions.  This is stated formally in Theorem \ref{asymptotic_linearity}.  

\begin{theorem}
\label{asymptotic_linearity}
\textup{\textbf{(Asymptotically linear expansion)}}\\
Let $\{ \hat{\mu}_{1,n} \}_{n \geq 1}$ and $\{ \hat{\mu}_{0,n} \}_{n \geq 1}$ satisfy the assumptions of Theorem \ref{consistency}, and further suppose that these models have typically simple realizations.  Then we have the following asymptotically linear expansions:
\begin{align}
\frac{1}{n} \sum_{i = 1}^n (\hat{y}_{1i} - y_{1i}) &= \frac{1}{n_1} \sum_{Z_i = 1} \epsilon_{1i}^* + o_p(n^{-1/2}) \label{expansion1}\\
\frac{1}{n} \sum_{i = 1}^n (\hat{y}_{0i} - y_{0i}) &= \frac{1}{n_0} \sum_{Z_i = 0} \epsilon_{0i}^* + o_p(n^{-1/2}) \label{expansion0}
\end{align}
where $\epsilon_{1i}^* := y_{1i} - \mu_{1,n}^*(\mathbf{x}_i)$ and $\epsilon_{0i}^* := y_{0i} - \mu_{0,n}^*(\mathbf{x}_i)$.  Moreover, $\mu_{1,n}^*$ and $\mu_{0,n}^*$ may be chosen so that $\tfrac{1}{n} \sum_{i = 1}^n \epsilon_{1i}^* = \tfrac{1}{n} \sum_{i = 1}^n \epsilon_{0i}^* = 0$.
\end{theorem}

In words, Theorem \ref{asymptotic_linearity} says that $\tfrac{1}{n} \sum_{i = 1}^n (\hat{y}_{ti} - y_{ti})$ is essentially just a sample average of some mean-zero constants.  Thus, under some Lindeberg-type conditions, the completely randomized central limit theorem \citep{li_and_ding} can be used to prove the joint asymptotic normality of $\tfrac{1}{n} \sum_{i = 1}^n (\hat{y}_{1i} - y_{1i})$ and $\tfrac{1}{n} \sum_{i = 1}^n (\hat{y}_{0i} - y_{0i})$, which implies the asymptotic normality of $(\hat{\tau}_n - \tau_n)$.  Corollary \ref{asymptotic_normality} gives some sufficient conditions.

\begin{corollary}
\label{asymptotic_normality}
\textup{\textbf{(Asymptotic normality)}}\\
Assume the conclusion of Theorem \ref{asymptotic_linearity}.  Further assume that $\mathsf{MSE}_n(t)$ is bounded away from zero and $\max_i (\epsilon_{ti}^*)^2 = o(n)$ for $t \in \{ 0, 1 \}$.  If the residual correlation $\rho_n := \langle \epsilon_1^*, \epsilon_0^* \rangle / (|| \epsilon_1^* ||_2 || \epsilon_0^* ||_2)$ is bounded away from $-1$, then we have:
\begin{align}
    \frac{\sqrt{n} (\hat{\tau}_n - \tau_n)}{\sigma_n} \xrightarrow{d} \mathcal{N}(0, 1)
\end{align}
where $\sigma_n^2 = \tfrac{1}{p_n} \mathsf{MSE}_n(1) + \tfrac{1}{1 - p_n} \mathsf{MSE}_n(0) - \tfrac{2}{n} \sum_{i = 1}^n (\epsilon_{1i}^* - \epsilon_{0i}^*)^2$
\end{corollary}

Corollary \ref{asymptotic_normality} holds for any choice of the sequences $\{ \mu_{1,n}^* \}_{n \geq 1}$ and $\{ \mu_{0,n}^* \}_{n \geq 1}$ that satisfy the stated conditions.  The requirement that $\mathsf{MSE}_n(0)$, $\mathsf{MSE}_n(1)$ are bounded away from zero and $\rho_n$ is bounded away from $-1$ are only used to rule out the degenerate situation where $\sqrt{n}( \hat{\tau}_n - \tau_n) = o(1)$.  In those cases, $\hat{\tau}_n$ is still a very good estimate of $\tau_n$, but confidence intervals may not have asymptotically valid coverage.  Although this assumption is not used in the works by \cite{lin2013} and \cite{bloniarz}, that seems to be an oversight.  In practice, users should check the $R^2$ from their regression models;  if they are both very close to one, there may be reason for concern.  

Unlike Lin's result (Theorem \ref{lins_result}), Corollary \ref{asymptotic_normality} does not contain any ``noninferiority" claim, i.e. there is no guarantee that a generalized Oaxaca-Blinder estimator is never worse than Neyman's unadjusted difference-of-means estimator.  That is the price of generality.  Since the class of regression methods to which Corollary \ref{asymptotic_normality} applies is so broad, it inevitably contains some bad apples.  In practice however, it is quite difficult to construct an explicit example where a regression method that is actually \textit{used} has worse performance than the unadjusted estimator.  For logistic regression, we have not been able to construct any such example despite some effort.

\subsection{Confidence intervals}

In order to use Corollary \ref{asymptotic_normality} to construct confidence intervals for $\tau_n$, it is necessary to construct an estimate of $\sigma_n$.  Although $\sigma_n$ is not (in general) identifiable from observed data, there is an identifiable upper bound which follows from the calculations by Neyman \citep{neyman}
\begin{align}
\label{neymans_upper_bound}
n \sigma_n^2 \leq \frac{1}{n_1} \mathsf{MSE}_n(1) + \frac{1}{n_0} \mathsf{MSE}_n(0).
\end{align}
One way to estimate the upper bound in (\ref{neymans_upper_bound}) is to first use \textit{in-sample} residual variance (\ref{insample_mse}) as an estimate of $\mathsf{MSE}_n(t)$, and then plug the MSE estimates back into (\ref{neymans_upper_bound}).  Then, confidence intervals may be constructed using the estimated upper bound.
\begin{align}
\label{insample_mse}
    \widehat{\mathsf{MSE}}_n(t) := \frac{1}{n_t - 1} \sum_{Z_i = t} [y_{ti} - \hat{\mu}_{t,n}(\mathbf{x}_i)]^2
\end{align}
Theorem \ref{confidence_intervals} says that this will work, as long as the residuals from the ``population" model $\mu_{1}^*$ and $\mu_0^*$ are not too heavy-tailed.

\begin{theorem}
\label{confidence_intervals}
\textup{\textbf{(Confidence intervals)}}\\
Assume the conditions of Theorem \ref{asymptotic_linearity} and Corollary \ref{asymptotic_normality}. If $\mathsf{MSE}_n(t)$ stays bounded and $\tfrac{1}{n} \sum_{i = 1}^n (\epsilon_{ti}^*)^4 = o(n)$ for $t \in \{ 0, 1 \}$, then we may construct an asymptotically valid confidence interval for $\tau$:
\begin{align}
\label{ci}
\liminf_{n \rightarrow \infty} \P \left( \tau_n \in \left[ \hat{\tau}_n \pm z_{1 - \alpha/2} \sqrt{ \frac{\widehat{\mathsf{MSE}}_n(1)}{n_1} + \frac{\widehat{\mathsf{MSE}}_n(0)}{n_0}} \right] \right) \geq 1 - \alpha
\end{align}
When the treatment has no effect (i.e. $y_{1i} = y_{0i}$ for all $i$) and the same method is used to estimate both $\hat{\mu}_1$ and $\hat{\mu}_0$ (e.g. both logistic regression with the same covariates), then the asymptotic coverage of the confidence interval is exactly $1 - \alpha$.
\end{theorem}

\subsection{Extensions}

Theorem \ref{asymptotic_linearity} has a host of other consequences, which we briefly allude to.  The tools in \cite{li2020rerandomization} can be used to study the asymptotic distribution of an asymptotically linear statistic under Mahalanobis rerandomization, and the delta method can be used to derive the asymptotic distribution of other smooth functions of the imputed means.  Therefore, it should be possible to use generalized Oaxaca-Blinder estimators in rerandomized designs, and also to study generalized Oaxaca-Blinder odds-ratio estimators.  We leave the technical details of those extensions for future work.

\section{Examples} \label{examples}

In this section, we give some examples of specific regression methods that satisfy the assumptions of stability, prediction unbiasedness, and typically simple realizations.  We have chosen examples that cover a variety of situations in which Lin's ``interactions" estimator might be deficient:  (i) binary outcomes;  (ii) count outcomes;  (iii) skewed outcomes;  and (iv) highly nonlinear relationships.

\subsection{Logistic regression}

With binary outcomes, OLS covariate adjustment is intuitively ``wrong."  For example, if the perspective on Lin's ``interaction" estimator presented in Section \ref{the_estimator} is taken up, imputation with linear models is unsatisfactory because values less than 0 or larger than 1 may be imputed.  This deficiency has led many authors to consider procedures similar to the logistic-regression-based Oaxaca-Blinder estimator \citep{freedman_logit, ding2018, firth_bennett, hansen_bowers}, but none have given a truly satisfying proof that covariate adjustment with logistic models is valid in the randomization model.  

The closest result we have seen is the work of Hansen and Bowers \citep{hansen_bowers}, which establishes the asymptotic normality of $\tfrac{1}{n} \sum_{i = 1}^n (\hat{y}_{1i} - y_{1i})$ when $\hat{y}_{1i}$ is imputed using logistic regression.  However, the authors assume stability instead of proving it from low-level assumptions on the finite population.  Theorem \ref{logistic_regression} goes beyond that and establishes the asymptotically linear expansion for $\tfrac{1}{n} \sum_{i = 1}^n (\hat{y}_{1i} - y_{1i})$ under only primitive assumptions on the population.  Asymptotic normality follows under the additional requirements of Corollary \ref{asymptotic_normality}.

\begin{theorem}
\label{logistic_regression}
\textup{\textbf{(Logistic regression)}}\\
Suppose that the potential outcomes $y_{1i}$ are binary and $\hat{\mu}_1$ is estimated using logistic regression.  Assume that, for all large $n$, there exists a vector $\theta_{1,n}^*$ solving the ``population" logistic regression problem (\ref{population_logit}).
\begin{align}
\label{population_logit}
    \theta_{1,n}^* = \argmin_{\theta} \left\{ \L_n(\theta) :- \frac{1}{n} \sum_{i = 1}^n - y_{1i} \mathbf{x}_i^{\top} \theta + \log(1 + e^{\theta^{\top} \mathbf{x}_i}) \right\}
\end{align}
Further suppose that $\tfrac{1}{n} \sum_{i = 1}^n || \mathbf{x}_i ||^4$ is uniformly bounded and $\nabla^2 \L_n(\theta_{1,n}^*) \succeq \lambda_{\min} \mathbf{I}_{d \times d}$ for some positive constant $\lambda_{\min}$ not depending on $n$.  Then the sequence $\{ \hat{\mu}_{1,n} \}_{n \geq 1}$ satisfies the conditions of Theorem \ref{asymptotic_linearity}.
\end{theorem}

At least qualitatively speaking, the assumptions of Theorem \ref{logistic_regression} cannot be improved (although the uniform bounds could, with additional effort, be relaxed to depend on $n$).  It is widely known that the logistic MLE fails to exist when the design matrix $\mathbf{X}$ is rank-deficient or the 1s and 0s among the outcomes can be perfectly separated by a hyperplane \citep{albert_anderson}.  If either of these properties hold in the population, then they necessarily hold in a subsample.  The existence of the population MLEs is, therefore, necessary to guarantee the existence of the \textit{sample} MLEs.  

However, it is not sufficient.  Consider for instance the following possibilities:
\begin{itemize}
\item \textit{Near-perfect separation}.\\
  If only a single exceptional point prevents the $y_{1i}$s from being perfectly separated by a hyperplane, then $\theta_{1}^*$ exists.  However, in $100(1 - p_n) \%$ of the possible realizations of the treatment assignments $(Z_1, \cdots, Z_n)$, the exceptional point will not be observed.  In those samples, $\hat{\theta}_1$ will fail to exist and $\hat{\tau}$ will not be well-defined.
  
\item \textit{Near-perfect collinearity}.\\
  If two columns of the design matrix $\mathbf{X}$ have perfect correlation except for a few exceptions, then in many realizations of the treatment assignments $(Z_1, \cdots, Z_n)$, the design matrix used to estimate $\hat{\mu}_1$ will be rank-deficient.  Again, $\hat{\theta}_1$ will fail to exist.  This might occur if one of the predictors is a very sparsely populated indicator variable.
\end{itemize}
The assumption that the Fisher information $\nabla^2 \L_n(\theta_{1,n}^*)$ is bounded away from zero is used to rule out such pathological populations.  Some indications that these assumptions are violated are (i) many fitted values very close to 0.00 or 1.00, and  (ii) ``unnaturally large" standard errors on the regression coefficients.  Chapter 6 of the textbook by Agresti \citep{agresti} has a more detailed discussion.

\subsection{Poisson regression}

Although OLS adjustment is not as obviously ``wrong" in the case of count outcomes as in the case of binary outcomes, the example presented in Section \ref{fatalities_example} shows that Poisson models can lead to substantial efficiency gains in such settings.  Theorem \ref{asymptotic_linearity} gives some sufficient conditions to justify the validity of the Poisson regression generalized Oaxaca-Blinder estimator.

\begin{theorem}
\label{poisson_regression}
\textup{\textbf{(Poisson regression)}}\\
Suppose that the potential outcomes $y_{1i}$ are nonnegative integers and $\hat{\mu}_1$ is estimated using Poisson regression.  Assume that, for all large $n$, there exists a vector $\theta_{1,n}^*$ solving the ``population" Poisson regression problem (\ref{poisson_mle}).
\begin{align}
\label{poisson_mle}
    \theta_{1,n}^* = \argmin_{\theta} \left\{ \L_n(\theta) := \frac{1}{n} \sum_{i = 1}^n [ -y_{1i} \mathbf{x}_i^{\top} \theta + \exp(\theta^{\top} \mathbf{x}_i)] \right\}
\end{align}
Let $\mathsf{MSE}_n(1) = \tfrac{1}{n} \sum_{i = 1}^n [y_{1i} - \exp(\theta_{1,n}^{* \top} \mathbf{x}_i)]^2$.  Assume that $|| \mathbf{x}_i ||, \tfrac{1}{n} \sum_{i = 1}^n y_{1i}^2$, $\mathsf{MSE}_n(1)$ are uniformly bounded and $\nabla^2 \L_n(\theta_{1,n}^*) \succeq \lambda_{\min} \mathbf{I}_{d \times d}$ for some positive constant $\lambda_{\min}$ not depending on $n$.  Then the sequence $\{ \hat{\mu}_{1,n} \}_{n \geq 1}$ satisfies the conditions of Theorem \ref{asymptotic_linearity}.
\end{theorem}

The conditions are a bit stronger than those required for logistic regression.  This is partly due to the nature of the data;  the boundedness of $\tfrac{1}{n} \sum_{i = 1}^n y_{1i}$ and $\mathsf{MSE}_n(1)$ are automatic with binary outcomes, but needs to be assumed for count outcomes.  In any case, the assumption that $\mathsf{MSE}_n(1)$ is bounded is needed in Theorem \ref{confidence_intervals}.  The requirement that covariates are bounded is more restrictive than what was assumed in Theorem \ref{logistic_regression}, but in Poisson regression it is typical to log-transform covariates (since the inverse link is exponential).  On the log scale, boundedness is more palatable.

\subsection{Transformed-outcome regression}

Our next example covers the case of skewed outcomes.  When the distribution of outcome variables has a long right tail, linear models typically fit better after a log transformation.  This was the modeling strategy chosen in the original paper by Oaxaca \citep{oaxaca}, which studied wage data.  However, since treatment effect estimates are more interpretable on the original scale, it is not enough just to use Lin's ``interactions" estimator with a log-transformed outcome and then report the coefficient $\hat{\tau}$ as the causal effect.  Even researchers willing to assume that $(\mathbf{x}_i, y_{0i}, y_{1i})$ are i.i.d. samples from a probability distribution $\P$ and $\E[ \log (y_{ti}) | \mathbf{x}_i] = \theta_t^{\top} \mathbf{x}_i$ cannot recover the average treatment effect $\E[ y_{1i} - y_{0i}]$ from $\theta_1, \theta_0$ alone.  Additional distributional assumptions are needed to relate averages on the log scale to averages on the original scale.  

There are two ways to get around making these assumptions.  One way is to estimate the regression coefficients $\hat{\theta}_0$, $\hat{\theta}_1$ on the log scale, and then use the generalized Oaxaca-Blinder estimator based on the debiased model $\hat{\mu}_t^{\mathsf{db}}$.  
\begin{align}
    \hat{\mu}_t^{\mathsf{db}}(\mathbf{x}) &= \exp( \hat{\theta}_t^{\top} \mathbf{x}) - \frac{1}{n_t} \sum_{Z_i = t} [\exp(\hat{\theta}_t^{\top} \mathbf{x}_i) - y_{ti}]
\end{align}

The other way (which we recommend) is to use the fitted values $\exp( \hat{\theta}_t^{\top} \mathbf{x})$ as a covariate in a second-stage OLS regression, and then use the generalized Oaxaca-Blinder estimator based on $\hat{\mu}_t^{\mathsf{ols2}}$.
\begin{align}
\hat{\mu}_t^{\mathsf{ols2}}(\mathbf{x}) &= \hat{\beta}_{0,t} + \hat{\beta}_{1,t} \exp( \hat{\theta}_t^{\top} \mathbf{x})
\end{align}

Theorem \ref{log_transform} says that under assumptions similar to those used in the case of logistic and Poisson regression, either of these strategies will ``work."

\begin{theorem}
\label{log_transform}
\textup{\textbf{(OLS with log-transformed outcome)}}\\
Assume that the $|| \mathbf{x}_i ||$, $\tfrac{1}{n} \sum_{i = 1}^n y_{1i}^2$, and $\tfrac{1}{n} \sum_{i = 1}^n (\log y_{1i})^4$ are uniformly bounded.  Further suppose that $\tfrac{1}{n} \mathbf{X}^{\top} \mathbf{X} \succeq \lambda_{\min} \mathbf{I}_{d \times d}$ for some positive constant $\lambda_{\min}$ not depending on $n$.  Then both the debiased estimator $\{ \hat{\mu}_{1,n}^{\mathsf{db}} \}_{n \geq 1}$ and the second-stage OLS estimator $\{ \hat{\mu}_{1,n}^{\mathsf{ols2}} \}_{n \geq 1}$ satisfy the conditions of Theorem \ref{asymptotic_linearity}.
\end{theorem}

In practice, we recommend the estimator based on $\hat{\mu}_t^{\mathsf{ols2}}$ over the estimator based on $\hat{\mu}_t^{\mathsf{db}}$.  The extra degree of freedom $\hat{\beta}_{1,t}$ can make a big difference in terms model fit.  For example, when the true data-generating process is $y_{ti} = \exp( \theta_t^{\top} \mathbf{x}_i) + \epsilon_i$ where $\epsilon_i \sim \N(0, \sigma^2)$, the optimal predictor is $\beta_1 \exp( \theta_t^{\top} \mathbf{x}_i)$ where $\beta_1 = \exp(\sigma^2 / 2)$.  

We illustrate the difference using an example.  We used the dataset and experimental set-up from Example \ref{fatalities_example} to compare the two transformed-outcome methods.  The outcome in this dataset (number of traffic fatalities) is highly skewed, and linear models fit much better after a log transformation of outcomes and covariates.  Figure \ref{fig:skew_example} plots the randomization distribution of the generalized Oaxaca-Blinder estimator based on (i) OLS regression without any transformations\footnote{This is the same as Lin's ``interactions" estimator.};  (ii) the debiased estimator $\hat{\mu}_t^{\mathsf{db}}$;  (iii) the second-stage OLS estimator $\hat{\mu}_t^{\mathsf{ols2}}$.  The estimator based on $\hat{\mu}_t^{\mathsf{ols2}}$ comes out as the clear winner in terms of precision.  This is reflected in the average width of the 95\% confidence intervals, which were (i) 106 deaths;  (ii) 84 deaths;  (iii) 78 deaths.  All three intervals had approximately nominal coverage

\begin{figure}[H]
    \centering
    \includegraphics[width = 12.5cm]{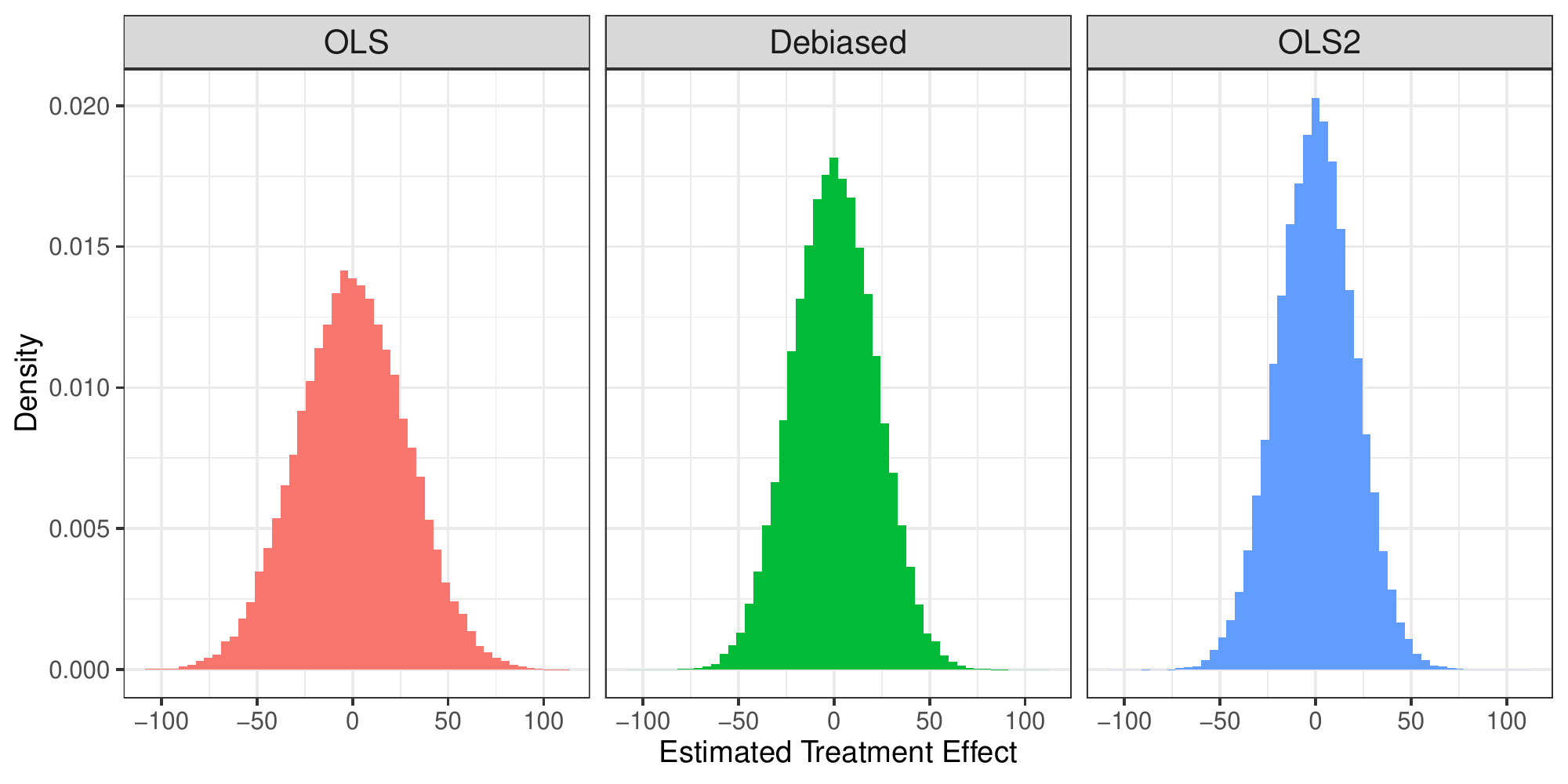}
    \caption{\textit{The randomization distribution of Lin's ``interactions" estimator (left), the generalized Oaxaca-Blinder estimator based on ``debiasing" the predictions $\exp( \hat{\theta}_t^{\top} \mathbf{x}_i)$ (center), and the generalized Oaxaca-Blinder estimator based on using the predictions $\exp( \hat{\theta}_t^{\top} \mathbf{x}_i)$ as a covariate in a linear model (right).  All models control for state population, average miles per driver, and per capita income.  The covariates are log transformed in the right two panels.}}
    \label{fig:skew_example}
\end{figure}

\subsection{Isotonic regression}

Our last example is nonparametric.  Suppose that the covariate dimension $d$ is equal to one, and the outcome $y_{1i}$ is bounded in the interval $[a, b]$.  When the relationship between $x_i$ and $y_{1i}$ is expected to be monotone increasing, \textit{isotonic regression} is a common modeling strategy.  This method finds a function $\hat{\mu}_1$ by solving:
\begin{align}
    \label{isotonic}
    \hat{\mu}_1 \in \argmin_{ \mu \in \mathsf{M}} \frac{1}{n_1} \sum_{Z_i = 1} (y_{1i} - \mu(x_i))^2
\end{align}
where $\mathsf{M} = \{ \mu : \R \rightarrow [a, b] \text{ nondecreasing} \}$.  Although the argmin in (\ref{isotonic}) is not unique, the value of $\hat{\mu}_1$ at the gridpoints is determined uniquely \citep{isotonic}.  For concreteness, we will choose the piecewise linear solution.

The theoretical properties of isotonic regression in the i.i.d. setting are well-studied (see the review paper \cite{shape_restricted}), but no previous works have studied this method from the perspective of randomization inference.  Theorem \ref{isotonic_regression} says that, under almost no assumptions, isotonic regression satisfies all of the conditions needed to use it in the generalized Oaxaca-Blinder procedure.

\begin{theorem}
\label{isotonic_regression}
\textup{\textbf{(Isotonic regression)}}\\
Let $x_i$ be a real-valued covariate, and $y_{1i} \in [a, b]$ a bounded outcome.  Let $\hat{\mu}_{1,n}$ be the piecewise-linear solution to (\ref{isotonic}).  Then the sequence $\{ \hat{\mu}_{1,n} \}_{n \geq 1}$ satisfies the conditions of Theorem \ref{asymptotic_linearity}.
\end{theorem}

Since the isotonic class $\mathsf{M}$ contains all constant functions, the (in-sample) mean-squared error from isotonic regression is never larger than the mean-squared error from a constant fit.  Therefore, the generalized Oaxaca-Blinder estimator based on two isotonic regressions shares ``non-inferiority" property of Lin's ``interactions" estimator:  the confidence interval will never be wider than Neyman's \citep{neyman} confidence interval for the difference-of-means estimator. 

That being said, we have found that using the in-sample prediction error as an estimate of $\mathsf{MSE}_n(t)$ performs poorly with isotonic regression, due to overfitting.  In a simple synthetic-data example shown in Figure \ref{fig:isotonic}, about 600 samples were needed per treatment arm before 95\% confidence intervals achieved $>$ 93\% coverage.  One explanation for why so many samples are needed is that the ``degrees of freedom" in the isotonic class (whatever that means) is too large -- see the jagged regression function in Figure \ref{fig:isotonic}.  The smoothed variants of isotonic regression that have been proposed in the literature (\cite{isotonic}) may perform better in smaller samples.

\begin{figure}[H]
    \centering
    \includegraphics[width=12.5cm]{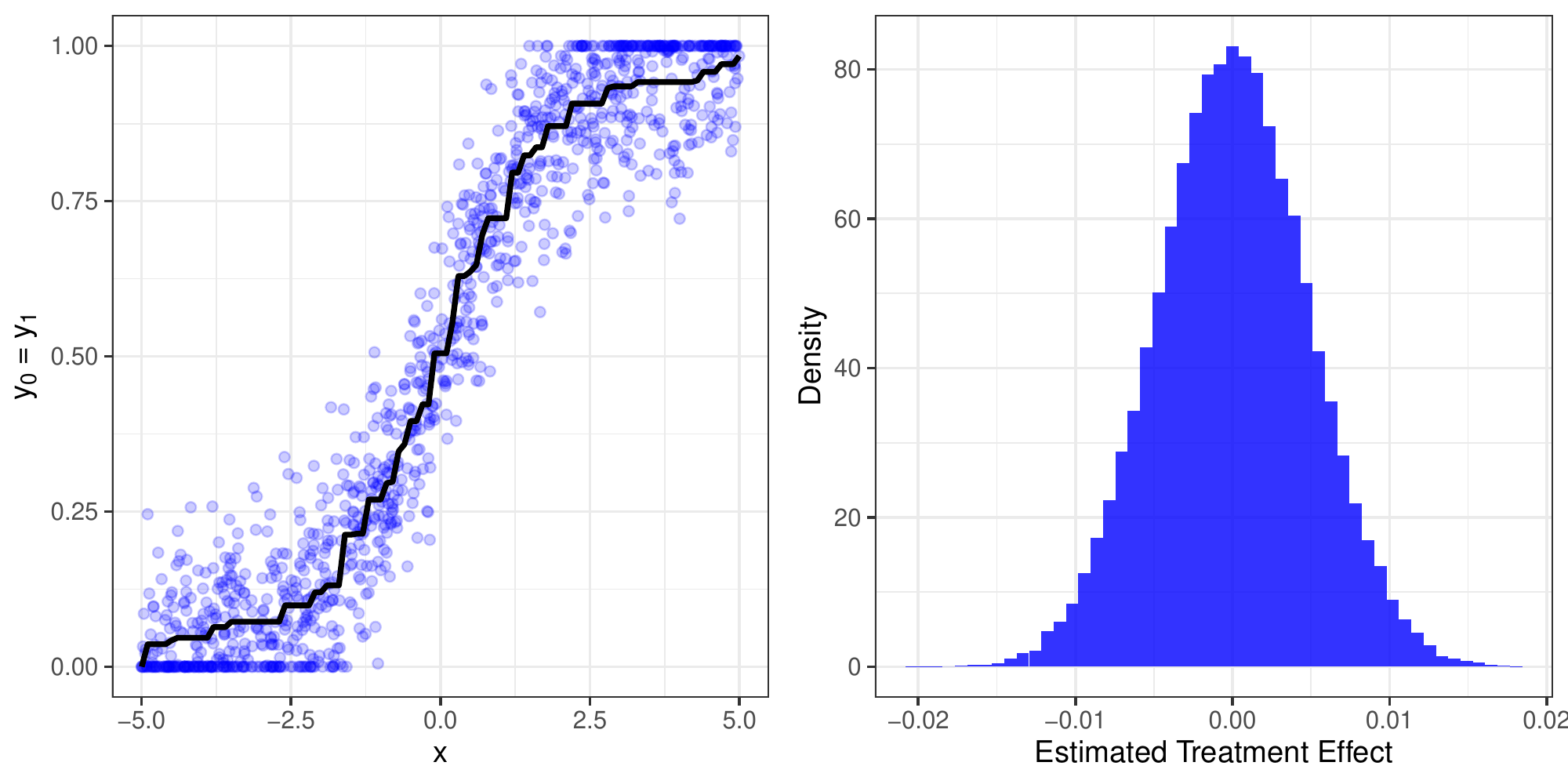}
    \caption{\textit{The ``population" isotonic regression fit (left), and the randomization distribution of the generalized Oaxaca-Blinder estimator based on two isotonic regressions (right).  There are 1200 observations in total, with $y_{0i} = y_{1i}$ for each $i$.  In each replication, half of the observations are assigned to the ``treatment" group and half of the observations are assigned to the ``control" group.}}
    \label{fig:isotonic}
\end{figure}

\section{Extension: a generic recipe for parametric models} \label{parametric_recipe}

The parametric examples presented in Section \ref{examples} can all be proved by translating standard arguments from the theory of M-estimation into the language of finite populations.  This section gives a ``master theorem" that lightens the effort of performing that translation in many cases. The result applies to parametric models where the parameter $\hat{\theta}$ is estimated by solving a convex M-estimation problem.

\begin{assumption}
\label{convex_m_estimator}
\textup{\textbf{(Convex loss function)}}\\
Assume that $\hat{\theta}_n$ is the solution to the following (random) optimization problem (\ref{m_estimation}).
\begin{align}
\label{m_estimation}
    \hat{\theta}_n \in \argmin_{\theta} \left\{ \hat{\L}_n(\theta) := \frac{1}{n_1} \sum_{Z_i = 1} \ell(\theta, \mathbf{x}_i, y_{1i}) \right\}
\end{align}
where $\ell(\theta, \mathbf{x}, y)$ is a loss function that is convex in its first argument.
\end{assumption}

The convexity condition is likely stronger than necessary. In the classical i.i.d. setting, consistency in convex M-estimation problems can be proved under essentially no assumptions \citep{niemiro1992}, and a similar phenomenon holds in the randomization setting.  In principle, our arguments (based on uniform convergence) can handle nonconvex loss functions as well, but solving the problem (\ref{m_estimation}) when the loss function $\hat{\L}_n$ is nonconvex is computationally challenging (except in highly specialized problems).  Therefore, theoretical results for nonconvex M-estimators are unlikely to have practical relevance and we have not pursued that direction.

The next assumption asks for the existence of a stable solution to the ``population" version of (\ref{m_estimation}).

\begin{assumption}
\label{growth}
\textup{\textbf{(Existence and stability of population optima)}}\\
Assume that, for all large $n$, there exists a vector $\theta_n^*$ solving the ``population" M-estimation problem (\ref{population_m_estimation}).
\begin{align}
\label{population_m_estimation}
    \theta_n^* = \argmin_{\theta} \left\{ \L_n(\theta) := \frac{1}{n} \sum_{i = 1}^n \ell(\theta, \mathbf{x}_i, y_{1i}) \right\}
\end{align}
Furthermore, suppose that there exists a strictly increasing function $f : \R_{\geq 0} \rightarrow \R_{\geq 0}$ so that $\L_n(\theta) - \L_n(\theta_n^*) \geq f (|| \theta - \theta_n^* ||)$ for all $\theta$.
\end{assumption}

The assumption that $\L_n(\theta) - \L_n(\theta_n^*) \geq f(|| \theta - \theta_n^* ||)$ guarantees the minimizer $\theta_n^*$ is unique and rules out population sequences where the objective function becomes flatter and flatter as $n \rightarrow \infty$.  A useful tool for verifying this assumption is the following Lemma.

\begin{lemma}
\label{growth_lemma}
\textup{\textbf{(Curvatures implies growth)}}\\
Suppose that the map $\theta \mapsto \ell(\theta, \mathbf{x}, y)$ is twice continuously differentiable for every $\mathbf{x}, y$ and the minimizer $\theta_n^*$ occurs at a point where $\nabla \L_n(\theta_n^*) = 0$.  Assume that for some radius $r > 0$ and some constant $\lambda_{\min} > 0$ not depending on $n$, we have $\nabla^2 \L_n(\theta) \succeq \lambda_{\min} \mathbf{I}_{d \times d}$ for all $\theta \in \mathbb{B}_r(\theta_n^*)$.  Then Assumption \ref{growth} is satisfied.
\end{lemma}

The final assumption needed for the convergence of $\hat{\theta}_n$ to $\theta_n^*$ is that $\L_n$ is smooth near its minimum.

\begin{assumption}
\label{smoothness}
\textup{\textbf{(Smooth loss function)}}\\
Assume that there exists a radius $r > 0$ and a constant $L < \infty$ (not depending on $n$) such that (\ref{local_smoothness}) holds for all $\theta, \phi \in \mathbb{B}_r(\theta_n^*)$.
\begin{align}
\label{local_smoothness}
\left( \frac{1}{n} \sum_{i = 1}^n [\ell(\theta, \mathbf{x}_i, y_{1i}) - \ell(\phi, \mathbf{x}_i, y_{1i})]^2 \right)^{1/2} \leq L || \theta - \phi ||
\end{align}
\end{assumption}

Under the above assumptions, $|| \hat{\theta}_n - \theta_n^* ||$ tends to zero in probability.  If the map $\theta \mapsto \mu_{\theta}$ is also smooth near $\theta_n^*$, then that implies stability and typically simple realizations.

\begin{theorem}
\label{general_parametric_models}
\textup{\textbf{(General results for parametric models)}}\\
Assume that \ref{convex_m_estimator}, \ref{growth}, and \ref{smoothness} are satisfied.  Then $|| \hat{\theta}_n - \theta_n^* || \xrightarrow{p} 0$.  If, in addition, $|| \mu_{\theta} - \mu_{\phi} ||_n \leq M || \theta - \phi ||$ for all $\theta, \phi \in \mathbb{B}_r(\theta_n^*)$, then the sequence $\{ \mu_{\hat{\theta}_n} \}_{n \geq 1}$ is stable and has typically simple realizations.
\end{theorem}


\section{Conclusion}

In this paper, we introduced an intuitive approach to performing covariate 
adjustment in randomized experiments. It can be summarized in a single sentence:  ``fill in the missing outcomes with an unbiased prediction model."  From a theoretical perspective, our main idea is that a little randomization goes a long way.  As long as treatment assignments are randomized, then tools from empirical process theory can be applied even if all other quantities are nonrandom.

Many open questions remain.  Perhaps the most important one is whether the Donsker-type ``typically simple realizations" assumption can be removed.  If one is willing to assume that $(\mathbf{x}_i, y_{0i}, y_{1i}, Z_i)$ are i.i.d. samples from a larger population, clever use of sample splitting can circumvent these assumptions \citep{double_machine_learning, wager_taylor_tibshirani_du}.  Are similar results available in Neyman's finite-population model?  Recent work \citep{loop} suggests that the answer might be ``yes'', but establishing the stability condition for these complex models remains challenging. We consider this to be an exciting direction for future work.

\newpage
\bibliographystyle{chicago}
\bibliography{biblio.bib}

\newpage
\section{Appendix} \label{appendix}

\subsection{Proofs}

\subsubsection{Randomization law of large numbers}

\begin{lemma}
\label{wlln}
Let $\mathcal{A}_n = \{ a_{i,n} \}_{i = 1}^n$ be a sequence of finite subsets of $\R$, and define $\bar{a}_n := \tfrac{1}{n} \sum_{i = 1}^n a_{i,n}$.  If $\tfrac{1}{n} \sum_{i = 1}^n a_{i,n}^2 = o(n)$, then we have:
\begin{align}
    \left| \frac{1}{n_1} \sum_{Z_i = 1} a_{i,n} - \bar{a}_n \right| \xrightarrow{p} 0
\end{align}
where $Z_i$ are defined in the main body of the paper.
\end{lemma}

\begin{proof}
By Proposition 1 in \cite{freedman_regression}, the variance of $\tfrac{1}{n_1} \sum_{Z_i = 1} a_{i,n}$ has a simple expression:
\begin{align*}
\Var \left( \frac{1}{n_1} \sum_{Z_i = 1} a_{i,n} \right) = \left( \frac{1 - p_n}{p_n} \right) \left( \frac{1}{n - 1} \right) \frac{1}{n} \sum_{i = 1}^n (a_{i,n} - \bar{a})^2
\end{align*}
Since $\tfrac{1}{n} \sum_{i = 1}^n (a_{i,n} - \bar{a})^2 \leq \tfrac{1}{n} \sum_{i = 1}^n a_{i,n}^2 = o(n)$, this variance tends to zero.  From this, $| \tfrac{1}{n_1} \sum_{Z_i = 1} a_{i,n} - \bar{a} | \rightarrow_p 0$ follows from Chebyshev's inequality.
\end{proof}

\subsubsection{A maximal inequality}

\begin{proposition}
\label{crd_maximal_inequality}
For any function $f : \R^d \rightarrow \R$, define $\mathbb{G}_n(f)$ as follows:
\begin{align}
    \mathbb{G}_n(f) = \frac{1}{\sqrt{n}} \sum_{i = 1}^n \left[ \frac{Z_i f(\mathbf{x}_i)}{p_n} - f(\mathbf{x}_i) \right] 
\end{align}
where $p_n = n_1 / n$.  If $\mathcal{F}$ is any collection of functions, $f_0 \in \mathcal{F}$ is any fixed function, and $\delta > 0$ is any radius, then we have the inequality:
\begin{align}
\label{maximal_inequality}
    \E \left[ \sup_{f \in \mathcal{F} \, : \, || f - f_0 ||_n \leq \delta} | \mathbb{G}_n(f) - \mathbb{G}_n(f_0)| \right] \leq (C / p_{\min}) \int_0^{\delta} \sqrt{\log \mathsf{N}( \mathcal{F}, || \cdot ||_n, s)} \, \d s
\end{align}
where $|| f - f_0 ||_n^2 := \tfrac{1}{n} \sum_{i = 1}^n [f(\mathbf{x}_{i,n}) - f_0(\mathbf{x}_{i,n})]^2$ and $C < \infty$ is a universal constant.
\end{proposition}

\begin{proof}
First, we show that $\{ \mathbb{G}_n(f) \, : \, f \in \mathcal{F} \}$ is a $(c / p_{\min})$-sub-Gaussian process indexed by the metric space $(\mathcal{F}, || \cdot ||_n)$.  Lemma A2 in Wu \& Ding \citep{wu2018randomization} gives the following exponential tail bound for the distance between sample averages and their population counterparts:
\begin{align*}
\P \left( \left| \frac{1}{n} \sum_{i = 1}^n\left[  \frac{Z_i f(\mathbf{x}_i)}{p_n} - f(\mathbf{x}_i)\right] \right| > t \right)  \leq 2 \exp \left( - \frac{n t^2}{2 \sigma^2(f) / p_n^2} \right)
\end{align*}
where $\sigma^2(f) = \tfrac{1}{n} \sum_{i = 1}^n [f(\mathbf{x}_i) - \bar{f}]^2$ and $p_n = n_1 / n$.  We may loosen the upper bound using the inequality $\sigma^2(f) / p_n^2 \leq || f ||_n^2 / p_{\min}^2$.  That implies the following tail bound for $\mathbb{G}_n(f)$:
\begin{align*}
    \P( | \mathbb{G}_n(f)| > t) \leq 2 \exp \left( -\frac{ t^2}{2 || f ||_n^2 / p_{\min}^2} \right)
\end{align*}
Since the above bound holds for any function $f$, it also holds for functions of the form $f = g - h$ with $g, h \in \mathcal{F}$.  Since $\mathbb{G}_n(g - h) = \mathbb{G}_n(g) - \mathbb{G}_n(h)$, this yields a sub-Gaussian tail bound for the increments of the process.
\begin{align*}
    \P( | \mathbb{G}_n(g) - \mathbb{G}_n(h)| > t) \leq 2 \exp \left( - \frac{t^2}{2 || g - h ||_n^2 / p_{\min^2}} \right)
\end{align*}
Hence, $\{ \mathbb{G}_n(f) \, : \, f \in \mathcal{F} \}$ is a sub-Gaussian process.  Moreover, this process is separable since the map $f \mapsto \mathbb{G}_n(f)$ is always continuous in the $|| \cdot ||_n$-norm on $\mathcal{F}$.  Therefore, by Dudley's entropy integral (Corollary 2.2.8 in \cite{vdv_wellner}), there exists some universal constant $C < \infty$ such that (\ref{maximal_inequality}) holds.
\end{proof}

\subsubsection{Population regression is always prediction unbiased}

\begin{lemma}
\label{population_prediction_unbiasedness}
\textup{\textbf{(Prediction unbiasedness in the population)}}\\
Let $\{ \hat{\mu}_{1,n} \}_{n \geq 1}$ be a sequence of prediction-unbiased regression functions.  Assume that there exists a nonrandom sequence $\{ \mu_{1,n}^* \}_{n \geq 1}$ such that $|| \hat{\mu}_{1,n} - \mu_{1,n}^* ||_n \xrightarrow{p} 0$ and $\tfrac{1}{n} \sum_{i = 1}^n [ \mu_{1,n}^*(\mathbf{x}_{i,n}) - y_{1i,n}]^2 = o(n)$.  Then it is always possible to choose the sequence $\{ \mu_{1,n}^* \}_{n \geq 1}$ so that (\ref{moment_matching}) holds for every $n$.
\begin{align}
    \label{moment_matching}
    \frac{1}{n} \sum_{i = 1}^n \mu_{1,n}^*(\mathbf{x}_{i,n}) = \frac{1}{n} \sum_{i = 1}^n y_{1i,n}
\end{align}
\end{lemma}

\begin{proof}
Let $\{ \mu_{1,n}^* \}_{n \geq 1}$ be any fixed sequence of functions satisfying $|| \hat{\mu}_{1,n} - \mu_{1,n}^* ||_n \xrightarrow{p} 0$ and $\tfrac{1}{n} \sum_{i = 1}^n [\mu_{1,n}^*(\mathbf{x}_i) - y_{1i}]^2 = o(n)$.  A simple calculation shows that the prediction bias must be vanishing:
\begin{align*}
   \left|  \frac{1}{n} \sum_{i = 1}^n [\mu_{1,n}^*(\mathbf{x}_i) - y_{1i}] \right| &= \left| \frac{1}{n_1} \sum_{Z_i = 1} [\mu_{1,n}^*(\mathbf{x}_i) - y_{1i}] \right| + o_p(1)\\
    &= \left| \frac{1}{n_1} \sum_{Z_i = 1} [ \mu_{1,n}^*(\mathbf{x}_i) - \hat{\mu}_{1,n}(\mathbf{x}_i)] \right| + o_p(1)\\
    &\leq \frac{1}{p} \frac{1}{n} \sum_{i = 1}^n | \mu_{1,n}^*(\mathbf{x}_i) - \hat{\mu}_{1,n}(\mathbf{x}_i)| + o_p(1) \\
    &\leq \frac{1}{p} || \mu_{1,n}^* - \hat{\mu}_{1,n} ||_n + o_p(1)
\end{align*}
The right-hand side of the above display is $o_p(1)$ by stability, and the left-hand side is nonrandom.  Therefore, it must be the case that $\tfrac{1}{n} \sum_{i = 1}^n [ \mu_{1,n}^*(\mathbf{x}_i) - y_{1i}] = o(1)$.  Then, we may define $\nu_{1,n}^*(\mathbf{x}) = \mu_{1,n}^*(\mathbf{x}) - \tfrac{1}{n} \sum_{i = 1}^n [\mu_{1,n}^*(\mathbf{x}_i) - y_{1i}]$ (i.e. subtract off the bias).  The property $|| \nu_{1,n}^* - \hat{\mu}_{1,n} ||_n \xrightarrow{p} 0$ still holds, as does $\tfrac{1}{n} \sum_{i = 1}^n [ \nu_{1,n}^*(\mathbf{x}_{i,n}) - y_{1i,n}]^2 = o(n)$.  Therefore, we may as well have chosen the original sequence $\mu_{1,n}^*$ so that (\ref{moment_matching}) is satisfied.
\end{proof}

\subsubsection{Proof of Theorem \ref{consistency}}

\begin{proof}
Let $\{ \mu_{1,n}^* \}_{n \geq 1}$ be the population prediction unbiased sequence which is guaranteed to exist by Lemma \ref{population_prediction_unbiasedness}.  We may write:
\begin{align*}
\left| \frac{1}{n} \sum_{i = 1}^n (\hat{y}_{1i} - y_{1i}) \right| &= \left| \frac{1}{n} \sum_{i = 1}^n [\hat{\mu}_{1,n}(\mathbf{x}_i) - y_{1i}] \right| \\
&\leq \left| \frac{1}{n} \sum_{i = 1}^n [\mu_{1,n}^*(\mathbf{x}_i) - y_{1i}]\right| + \left| \frac{1}{n} \sum_{i = 1}^n [\hat{\mu}_{1,n}(\mathbf{x}_i) - \mu_{1,n}^*(\mathbf{x}_i)] \right| \\
&\leq 0 + || \hat{\mu}_{1,n} - \mu_{1,n}^* ||_n\\
&= o_p(1)
\end{align*}
A symmetric argument shows that $\tfrac{1}{n} \sum_{i = 1}^n (\hat{y}_{0i} - y_{0i}) = o_p(1)$.  Thus, $\hat{\tau}_n - \tau_n = o_p(1)$.
\end{proof}

\subsubsection{Proof of Theorem \ref{asymptotic_linearity}}

\begin{proof}
We start by proving (\ref{expansion1}).  Let $ \{ \mu_{1,n}^* \}_{n \geq 1}$ be a sequence of nonrandom functions such that $|| \hat{\mu}_{1,n} - \mu_{1,n}^* ||_n \rightarrow_p 0$ and $\tfrac{1}{n} \sum_{i = 1}^n \mu_{1,n}^*(\mathbf{x}_i) = \tfrac{1}{n} \sum_{i = 1}^n y_{1i}$.  Such a sequence is always guaranteed to exist (Lemma \ref{population_prediction_unbiasedness}).  By rearranging terms, we may write:
\begin{align*}
\frac{1}{n} \sum_{i = 1}^n ( \hat{y}_{1i} - y_{1i}) &= \frac{1}{n} \sum_{i = 1}^n [\hat{\mu}_{1,n}(\mathbf{x}_i) - y_{1i}] + \frac{1}{n} \sum_{i = 1}^n \frac{Z_i [y_{1i} - \hat{\mu}_{1,n}(\mathbf{x}_i)]}{p_n}\\
&=\frac{1}{n_1} \sum_{Z_i = 1} \epsilon_{1i}^* + \underbrace{\frac{1}{n} \sum_{i = 1}^n \left[ \frac{Z_i \mu_{1,n}^*(\mathbf{x}_i)}{p_n} - \mu_{1,n}^*(\mathbf{x}_i) \right]}_{\sqrt{n} \mathbb{G}_n(\mu_{1,n}^*)}  - \underbrace{\frac{1}{n} \sum_{i = 1}^n \left[ \frac{Z_i \hat{\mu}_{1,n}(\mathbf{x}_i)}{p_n} - \hat{\mu}_{1,n}(\mathbf{x}_i) \right]}_{\sqrt{n} \mathbb{G}_n( \hat{\mu}_{1,n})}
\end{align*}
The quantity $\mathbb{G}_n( \mu_{1,n}^*) - \mathbb{G}_n( \hat{\mu}_{1,n})$ is vanishing in probability.  To see this, let $\epsilon, \delta > 0$ be arbitrary.  For every $r > 0$, Proposition \ref{crd_maximal_inequality} and the ``typically simple realizations" assumption imply the following tail bound:
\begin{align*}
\P \left( \sup_{\mu \in \mathcal{F}_n, \: \, || \mu - \mu_{1,n}^* ||_n \leq r} | \mathbb{G}_n(\mu) - \mathbb{G}_n(\mu_{1,n}^*) | > \epsilon \right) &\leq \frac{C}{p_{\min} \epsilon} \int_0^r \sqrt{\log \mathsf{N}( \mathcal{F}_n \cup \{ \mu_{1,n}^* \}, || \cdot ||_n, s)} \, \d s\\
&\leq \frac{C}{p_{\min} \epsilon} \int_0^r \sqrt{\log(1 + \mathsf{N}( \mathcal{F}_n, || \cdot ||_n, s)} \, \d s\\
&\leq \frac{C}{p_{\min} \epsilon} \int_0^r 1 + \sqrt{\log \mathsf{N}( \mathcal{F}_n, || \cdot ||_n, s)} \, \d s\\
&\leq \frac{C r}{p_{\min} \epsilon} + \frac{C}{p_{\min}} \int_0^r \sup_{n \geq 1} \sqrt{\log \mathsf{N}(\mathcal{F}_n, || \cdot ||_n, s)} \, \d s
\end{align*}
In the second-to-last step, we used the fact that $\log (1 + \mathsf{N}) \leq 1 + \log \mathsf{N}$ since the covering number is always at least one.  The upper bound is vanishing as $r \downarrow 0$ so there exists $r^*$ sufficiently small so that the upper bound is less than $\delta/3$.

For large enough $n$, $\P_{n_1,n}( \hat{\mu}_n \in \mathcal{F}_n)$ and $\P_{n_1,n}( || \hat{\mu}_n - \mu_{1,n}^* ||_n \leq r^*)$ are both at least $1 - \delta/3$.  Thus, with probability at least $1 - \delta$, $| \mathbb{G}_n(\hat{\mu}_{1,n}) - \mathbb{G}_n( \mu_{1,n}^*)| \leq \epsilon$.  Since $\epsilon, \delta$ are arbitrary, this shows that $| \mathbb{G}_n(\mu_{1,n}^*) - \mathbb{G}_n( \hat{\mu}_{1,n})| = o_p(1)$.  Therefore, $\tfrac{1}{n} \sum_{i = 1}^n (\hat{y}_{1i} - y_{1i}) = \tfrac{1}{n_1} \sum_{Z_i = 1} \epsilon_{1i}^* + o_p(n^{-1/2})$.
\end{proof}

\subsubsection{Proof of Theorem \ref{confidence_intervals}}

\begin{proof}
By the reverse triangle inequality, we have:
\begin{align*}
\left|  \left( \widehat{\mathsf{MSE}}_n(1) \right)^{1/2} - \left( \frac{1}{n_1} \sum_{Z_i = 1} (\epsilon_{1i}^*)^2 \right)^{1/2}  \right| &\leq \left( \frac{1}{n_1} \sum_{Z_i = 1} [\hat{\mu}_{1,n}(\mathbf{x}_i) - \mu_{1,n}^*(\mathbf{x}_i)]^2 \right)^{1/2}\\
&\leq \left( \frac{1}{n_1} \sum_{i = 1}^n [ \hat{\mu}_{1,n}(\mathbf{x}_i) - \mu_{1,n}^*(\mathbf{x}_i)]^2 \right)^{1/2}\\
&= \frac{1}{\sqrt{p_n}} || \hat{\mu}_{1,n} - \mu_{1,n}^* ||_n\\
&= o_p(1)
\end{align*}
Since $\tfrac{1}{n_1} \sum_{Z_i = 1} (\epsilon_{1i}^*)^2 \leq p_n^{-1} \mathsf{MSE}_n(1) = \mathcal{O}(1)$, convergence of square roots implies the convergence of the left-hand side of the above display without the square roots.  By the fourth moment assumption on $\epsilon_{1i}^*$ and the randomization law of large numbers (Lemma \ref{wlln}), $| \tfrac{1}{n_1} \sum_{Z_i = 1} (\epsilon_{1i}^*)^2 - \mathsf{MSE}_n(1)| \xrightarrow{p} 0$.  Therefore, $| \widehat{\mathsf{MSE}}_n(1) - \mathsf{MSE}_n(1)| \xrightarrow{p} 0$. A symmetric argument shows that $|\widehat{\mathsf{MSE}}_n(0) - \mathsf{MSE}_n(0)| \xrightarrow{p} 0$.

The assumptions of Corollary \ref{asymptotic_normality} imply that $n [ \tfrac{1}{n_1} \mathsf{MSE}_n(1) + \tfrac{1}{n_0} \mathsf{MSE}_n(0)]$ is bounded away from zero.  Therefore, consistent estimation of $\mathsf{MSE}_n(1)$ and $\mathsf{MSE}_n(0)$ implies the following:
\begin{align*}
    \frac{\tfrac{1}{n_1} \widehat{\mathsf{MSE}}_n(1) + \tfrac{1}{n_0} \widehat{\mathsf{MSE}}_n(0)}{\tfrac{1}{n_1} \mathsf{MSE}_n(1) + \tfrac{1}{n_0} \mathsf{MSE}_n(0)} \xrightarrow{p} 1
\end{align*}
The asymptotic validity of the confidence intervals then follows from Slutsky's theorem.  In the case where the treatment has no effect and $\hat{\mu}_1$ and $\hat{\mu}_0$ are estimated by the same method, then $\mu_{1,n}^* = \mu_{0,n}^*$ so $\epsilon_{1i}^* - \epsilon_{0i}^* = 0$.  This means that the upper bound on $\sigma_n$ is exactly $\sigma_n$, and again Slutsky's theorem implies asymptotically exact coverage. 
\end{proof}

\subsubsection{Proof of Theorem \ref{logistic_regression}}

\begin{proof}
We will use Theorem \ref{general_parametric_models} to prove the result.  Define $\psi(s) = \log(1 + e^s)$.  Assumption \ref{convex_m_estimator} is satisfied, because the loss function is $\ell(\theta, \mathbf{x}, y) = -y \mathbf{x}^{\top} \theta + \psi(\theta^{\top} \mathbf{x})$, which is convex in $\theta$.  The existence of a population solution is assumed, so the first half of assumption \ref{growth} does not require any calculation to verify.  However, the growth condition requires an argument.  Let $L$ be a uniform bound on $\tfrac{1}{n} \sum_{i = 1}^n || \mathbf{x}_i ||^4$.  Since $\ddot{\psi}$ is a 1-Lipschitz function, the map $\theta \mapsto \nabla^2 \L_n(\theta)$ is an $L$-Lipschitz function.
\begin{align*}
|| \nabla^2 \L_n(\theta) - \nabla^2 \L_n(\theta_{1,n}^* ||_{\mathsf{op}} &= \left| \left| \frac{1}{n} \sum_{i = 1}^n ( \ddot{\psi}(\theta^{\top} \mathbf{x}_i) - \ddot{\psi}(\theta_{1,n}^{* \top} \mathbf{x}_i) \mathbf{x}_i \mathbf{x}_i^{\top} \right| \right|_{\mathsf{op}}\\
&\leq \frac{1}{n} \sum_{i = 1}^n | \ddot{\psi}(\theta^{\top} \mathbf{x}_i) - \ddot{\psi}(\theta_{1,n}^{* \top} \mathbf{x}_i) | \cdot || \mathbf{x}_i \mathbf{x}_i^{\top} ||_{\mathsf{op}}\\
&\leq \frac{1}{n} \sum_{i = 1}^n | (\theta - \theta_{1,n}^*)^{\top} \mathbf{x}_i| \cdot || \mathbf{x}_i ||^2\\
&\leq \left( \frac{1}{n} \sum_{i = 1}^n || \mathbf{x}_i ||^3 \right) || \theta - \theta_{1,n}^* ||\\
&\leq L || \theta - \theta_{1,n}^* ||
\end{align*}
This implies that for all $\theta$ in a ball of radius $\tfrac{1}{2} \lambda_{\min} / L$ around $\theta_{1,n}^*$, $\nabla^2 \L_n(\theta) \succeq \tfrac{1}{2} \lambda_{\min} \mathbf{I}_{d \times d}$.
\begin{align*}
\Lambda_{\min} ( \nabla^2 \L_n(\theta)) &= \inf_{|| v || = 1} v^{\top} \nabla^2 \L_n(\theta) v\\
&\geq \inf_{|| v || = 1} v^{\top} \nabla^2 \L_n (\theta_{1,n}^*) v - \sup_{|| v || = 1} v^{\top} [\nabla^2 \L_n(\theta) - \nabla^2 \L_n(\theta_{1,n}^*)] v\\
&\geq \lambda_{\min} - || \nabla^2 \L_n(\theta) - \nabla^2 \L_n(\theta_{1,n}^*) ||_{\mathsf{op}}\\
&\geq \lambda_{\min} - L || \theta - \theta_{1,n}^* ||
\end{align*}
The lower bound is at least $\tfrac{1}{2} \lambda_{\min}$ since $L || \theta - \theta_{1,n}^* || \leq \tfrac{1}{2} \lambda_{\min}$.  From this, assumption \ref{growth} follows from Lemma \ref{growth_lemma}.

Next, we verify assumption \ref{smoothness}.  This is a straightforward consequence of the inequality $(a - b)^2 \leq 2a^2 + 2b^2$ and the fact that $\psi$ is a 1-Lipschitz function.
\begin{align*}
|| \ell_{\theta} - \ell_{\phi} ||_n^2 &\leq \frac{2}{n} \sum_{i = 1}^n [\mathbf{x}_i^{\top} (\phi - \theta)]^2 + \frac{2}{n} \sum_{i = 1}^n (\psi(\theta^{\top} \mathbf{x}_i) - \psi(\phi^{\top} \mathbf{x}_i))^2\\
&\leq \frac{2}{n} \sum_{i = 1}^n || \mathbf{x}_i ||^2 || \phi - \theta ||^2 + \frac{2}{n} \sum_{i = 1}^n || \mathbf{x}_i^{\top} (\theta - \phi) ||^2\\
&\leq 4 L || \theta - \phi ||^2
\end{align*}
Thus, $|| \ell_{\theta} - \ell_{\phi} ||_n \leq 2 \sqrt{L} || \theta - \phi ||$.

Finally, we need to check that $|| \mu_{\theta} - \mu_{\phi} ||_n \leq M || \theta - \phi ||_n$ for some constant $M < \infty$.  This is simple, since $\dot{\psi}$ is also a 1-Lipschitz function.
\begin{align*}
|| \mu_{\theta} - \mu_{\phi} ||_n^2 &= \frac{1}{n} \sum_{i = 1}^n [\dot{\psi}(\theta^{\top} \mathbf{x}_i) - \dot{\psi}( \phi^{\top} \mathbf{x}_i)]^2\\
&\leq \frac{1}{n} \sum_{i = 1}^n [(\theta - \phi)^{\top} \mathbf{x}_i]^2\\
&\leq L || \theta - \phi ||^2
\end{align*}
Thus, we may take $M = \sqrt{L}$.  Since all the conditions of Theorem \ref{general_parametric_models} are satisfied, the conclusion follows.
\end{proof}

\subsubsection{Proof of Theorem \ref{poisson_regression}}

\begin{proof}
Once again, we will use Theorem \ref{general_parametric_models}.  Assumption \ref{convex_m_estimator} is satisfied because the loss function $\ell(\theta, \mathbf{x}, y) = - y \mathbf{x}^{\top} \theta + \exp(\theta^{\top} \mathbf{x})$ is convex in $\theta$.  The existence requirement of Assumption \ref{growth} is a condition of the theorem, and the ``growth" condition will be verified using the same strategy that was used in the proof of Theorem \ref{logistic_regression}.  Let $L < \infty$ be a uniform bound on $|| \mathbf{x}_i ||$ and $\tfrac{1}{n} \sum_{i = 1}^n y_{1i}^2$.  For any $\theta \in \mathbb{B}_1(\theta_{1,n}^*)$, the triangle inequality and the mean-value theorem imply:
\begin{align*}
|| \nabla \L_n(\theta) - \nabla^2 \L_n(\theta_{1,n}^*) ||_{\mathsf{op}} &= \left| \left| \frac{1}{n} \sum_{i = 1}^n \left( e^{\theta^{\top} \mathbf{x}_i} - e^{\theta_{1,n}^{* \top} \mathbf{x}_i} \right) \mathbf{x}_i \mathbf{x}_i^{\top} \right| \right|_{\mathsf{op}}\\
&\leq \frac{1}{n} \sum_{i = 1}^n \left| e^{\theta^{\top} \mathbf{x}_i} - e^{\theta_{1,n}^{* \top} \mathbf{x}_i} \right| \cdot || \mathbf{x}_i \mathbf{x}_i^{\top} ||_{\mathsf{op}}\\
&= \frac{1}{n} \sum_{i = 1}^n e^{c_i} | \theta^{\top} 
\mathbf{x}_i - \theta_{1,n}^{* \top} \mathbf{x}_i | \cdot || \mathbf{x}_i ||^2
\end{align*}
In the above display, $c_i$ is some number between $\theta^{\top} \mathbf{x}_i$ and $\theta_{1,n}^{* \top} \mathbf{x}_i$.  Since $\theta$ is within distance one of $\theta_{1,n}^*$ and $|| \mathbf{x}_i ||$ is bounded by $L$, $|c_i - \theta_{1,n}^{* \top} \mathbf{x}_i| \leq | \theta^{\top} \mathbf{x}_i - \theta_{1,n}^{* \top} \mathbf{x}_i| \leq L$.  In particular, $c_i \leq \theta^{\top} \mathbf{x}_i + L$.  We may then further bound the above by:
\begin{align*}
|| \nabla \L_n(\theta) - \nabla^2 \L_n(\theta_{1,n}^*) ||_{\mathsf{op}} &\leq \frac{1}{n} \sum_{i = 1}^n e^{\theta_{1,n}^{* \top} \mathbf{x}_i + L} || \theta - \theta_{1,n}^* || \cdot || \mathbf{x}_i ||^3\\
&\leq L^3 e^L \left( \frac{1}{n} \sum_{i = 1}^n e^{\theta_{1,n}^{* \top} \mathbf{x}_i} \right) || \theta - \theta_{1,n}^* ||\\
&= L^3 e^L \left( \frac{1}{n} \sum_{i = 1}^n y_{1i} \right) || \theta - \theta_{1,n}^* ||\\
&\leq L^{3.5} e^L || \theta - \theta_{1,n}^* ||
\end{align*}
Set $M = \min \{ \tfrac{1}{2} \lambda_{\min} / [L^{3.5} e^L], 1 \}$.  For any $\theta \in \mathbb{B}_{M}(\theta_{1,n}^*)$, we have $\Lambda_{\min} ( \nabla^2 \L_n(\theta)) \geq \Lambda_{\min} ( \nabla^2 \L_n(\theta_{1,n}^*)) - || \nabla^2 \L_n(\theta) - \nabla^2 \L_n(\theta_{1,n}^*) ||_{\mathsf{op}} \geq \tfrac{1}{2} \lambda_{\min}$, so Lemma \ref{growth_lemma} can be used to verify Assumption \ref{growth}.  

Next, we check Assumption \ref{smoothness}.  Let $\theta, \phi \in \mathbb{B}_1(\theta_{1,n}^*)$ be arbitrary.  Again by the mean-value theorem, there exists vectors $\bar{\theta}_1, \cdots, \bar{\theta}_n \in \mathbb{B}_1(\theta_{1,n}^*)$ such that the following calculations hold:
\begin{align*}
\frac{1}{n} \sum_{i = 1}^n [\ell(\theta, \mathbf{x}, y) - \ell(\phi, \mathbf{x}_i, y_{1i})]^2 &= \frac{1}{n} \sum_{i = 1}^n \langle \nabla_{\theta} \ell(\bar{\theta}_i, \mathbf{x}_i, y_{1i}) , \theta - \phi \rangle^2\\
&\leq \left( \frac{1}{n} \sum_{i = 1}^n || \nabla_{\theta} \ell(\bar{\theta}_i, \mathbf{x}_i, y_{1i}) ||^2 \right) || \theta - \phi ||^2\\
&= \left( \frac{1}{n} \sum_{i = 1}^n || - y_{1i} \mathbf{x}_i + \mathbf{x}_i \exp( \bar{\theta}_i^{\top} \mathbf{x}_i) ||^2 \right) || \theta - \phi ||^2\\
&\leq \left( \frac{1}{n} \sum_{i = 1}^n y_{1i}^2 || \mathbf{x}_i ||^2 +  \frac{1}{n} \sum_{i = 1}^n || \mathbf{x}_i ||^2 ( \exp(\bar{\theta}_i^{\top} \mathbf{x}_i))^2 \right) || \theta - \phi ||^2\\
&\leq \left( L^3 + L^2 \frac{1}{n} \sum_{i = 1}^n \exp( \bar{\theta}_i^{\top} \mathbf{x}_i))^2 \right)  || \theta - \phi ||^2
\end{align*}
It only remains to show that $\tfrac{1}{n} \sum_{i = 1}^n (\exp(\bar{\theta}_i^{\top} \mathbf{x}_i))^2$ is bounded.  Since $\bar{\theta}_i$ is between $\theta$ and $\phi$ which are both in $\mathbb{B}_1(\theta_{1,n}^*)$, $|| \bar{\theta}_i - \theta_{1,n}^* || \leq 1$.  Therefore, $\exp( \bar{\theta}_i^{\top} \mathbf{x}_i) \leq \exp( \theta_{1,n}^{* \top} \mathbf{x}_i + L)$, and it suffices to show that $\tfrac{1}{n} \sum_{i = 1}^n \exp(\theta_{1,n}^{* \top} \mathbf{x}_i)^2$ is bounded.  However, this follows immediately from the fact that $\tfrac{1}{n} \sum_{i = 1}^n y_{1i}^2$ and $\mathsf{MSE}_n(1)$ are both bounded.

Finally, we need to prove that $|| \mu_{\theta} - \mu_{\phi} ||_n \leq M || \theta - \phi ||$ for some $M < \infty$.  This follows by essentially the same argument as above.  For any $\theta, \phi \in \mathbb{B}_1(\theta_{1,n}^*)$, we have:
\begin{align*}
|| \mu_{\theta} - \mu_{\phi} ||_n^2 &= \frac{1}{n} \sum_{i = 1}^n |\exp(\theta^{\top} \mathbf{x}_i) - \exp(\phi^{\top} \mathbf{x}_i)|^2\\
&= \frac{1}{n} \sum_{i = 1}^n [\exp( \bar{\theta}_i^{\top} \mathbf{x}_i) (\mathbf{x}_i^{\top} \theta - \mathbf{x}_i^{\top} \phi)]^2\\
&\leq \left( L^2 \frac{1}{n} \sum_{i = 1}^n \exp( \bar{\theta}_i^{\top} \mathbf{x}_i)^2 \right) || \theta - \phi ||^2
\end{align*}
By the above argument, $\tfrac{1}{n} \sum_{i = 1}^n \exp( \bar{\theta}_i^{\top} \mathbf{x}_i)^2$ is bounded.
\end{proof}

\subsubsection{Proof of Theorem \ref{log_transform}}

\begin{lemma}
\label{log_ols_consistency}
\textup{\textbf{(Consistency of $\hat{\theta}_1$)}}\\
Let $\hat{\theta}_{1,n} = \argmin_{\theta} \sum_{Z_i = 1} [\log(y_{1i}) - \theta^{\top} \mathbf{x}_i]^2$ and let $\theta_{1,n}^* = \argmin_{\theta} \sum_{i = 1}^n [\log (y_{1i}) - \theta^{\top} \mathbf{x}_i]^2$.  Then $|| \hat{\theta}_{1,n} - \theta_{1,n}^* || \xrightarrow{p} 0$.
\end{lemma}

\begin{proof}
We will use Theorem \ref{general_parametric_models}.  The loss function is $\ell(\theta, \mathbf{x}, y) = [\log(y) - \theta^{\top} \mathbf{x}]^2$, which is convex in $\theta$.  The existence of $\theta_{1,n}^*$ is guaranteed by the assumption that $\tfrac{1}{n} \mathbf{X}^{\top} \mathbf{X}$ is invertible, and the growth condition is verified by the assumption that $\nabla^2 \L_n(\theta) = \tfrac{1}{n} \mathbf{X}^{\top} \mathbf{X} \succeq \lambda_{\min} \mathbf{I}_{d \times d}$.  To check Assumption \ref{smoothness}, we need to do some calculations.  First, we show that $|| \theta_{1,n}^* ||$ is bounded.  Let $L$ be a uniform bound on $|| \mathbf{x}_i ||^2$ and $\tfrac{1}{n} \sum_{i = 1}^n \log(y_{1i})^2$.  Then we have:
\begin{align*}
|| \theta_{1,n}^* || &= || (\mathbf{X}^{\top} \mathbf{X})^{-1} \mathbf{X}^{\top} \log(\mathbf{y}_1) || \leq \lambda_{\min}^{-1} \left( \frac{1}{n^2} \sum_{j = 1}^d [\mathbf{X}_{\bullet j} \log(\mathbf{y}_1)]^2 \right)^{1/2} \leq L / \lambda_{\min}
\end{align*}
Then for any $\theta, \phi \in \mathbb{B}_1(\theta_{1,n}^*)$, the mean-value theorem allows us to write:
\begin{align*}
\frac{1}{n} \sum_{i = 1}^n [ \ell(\theta, \mathbf{x}_i, y_{1i})) - \ell(\phi, \mathbf{x}_i, y_{1i}) ]^2 &= \frac{1}{n} \sum_{i = 1}^n [ 2( \log y_{1i} - \bar{\theta}_i^{\top} \mathbf{x}_i) \mathbf{x}_i^{\top} (\theta - \phi)]^2\\
&\leq 4 L \left( \frac{1}{n} \sum_{i = 1}^n [ \log y_{1i} - \bar{\theta}_i^{\top} \mathbf{x}_i]^2 \right)  || \theta - \phi ||^2\\
&\leq 8 L \left( \frac{1}{n} \sum_{i = 1}^n (\log y_{1i})^2 + \frac{1}{n} \sum_{i = 1}^n L || \bar{\theta}_i ||^2 \right) || \theta - \phi ||^2\\
&\leq 8 L [ L + L (L/\lambda_{\min} + 1)^2] \cdot || \theta - \phi ||^2
\end{align*}
Hence, all the assumptions of Theorem \ref{general_parametric_models} are satisifed, so $|| \hat{\theta}_{1,n} - \theta_{1,n}^* || \xrightarrow{p} 0$.
\end{proof}

\begin{lemma}
\label{exponentiated_stability}
\textup{\textbf{(Stability of the exponentiated model)}}\\
We have $|| \exp( \langle \hat{\theta}_{1,n} , \cdot \rangle) - \exp( \langle \theta_{1,n}^*, \rangle) ||_n \xrightarrow{p} 0$.
\end{lemma}

\begin{proof}
We use the fact that $|e^a - e^b| \leq e^c | a - b|$ if $|a|, |b| \leq c$.  Since with probability tending to one, $| \hat{\theta}_{1,n}^{\top} \mathbf{x}_i| \leq \sqrt{L} (L / \lambda_{\min} + 1) \equiv L'$ for all $i$, we may write:
\begin{align*}
\frac{1}{n} \sum_{i = 1}^n [ \exp( \hat{\theta}_{1,n}^{\top} \mathbf{x}_i) - \exp( \theta_{1,n}^{* \top} \mathbf{x}_i)]^2 &\leq \left( \frac{1}{n} \sum_{i = 1}^n e^{2L'} || \mathbf{x}_i ||^2 \right) || \hat{\theta}_{1,n} - \theta_{1,n}^* ||^2 + o_p(1)
\end{align*}
The upper bound is vanishing because $|| \mathbf{x}_i ||^2$ is bounded while $|| \hat{\theta}_{1,n} - \theta_{1,n}^* ||$ tends to zero by Lemma \ref{log_ols_consistency}.
\end{proof}

\begin{lemma}
\label{debiased_intercept_consistency}
\textup{\textbf{(Consistency of the debiased intercept)}}\\
Define $a_n^* = \tfrac{1}{n} \sum_{i = 1}^n [ \exp(\theta_{1,n}^{* \top} \mathbf{x}_i) - y_{1i}]$.  Define $\hat{a}_n = \tfrac{1}{n_1} \sum_{Z_i = 1} [ \exp(\hat{\theta}_{1,n}^{\top} \mathbf{x}_i) - y_{1i}]$.  Then $| \hat{a}_n - a_n^* | \xrightarrow{p} 0$.
\end{lemma}

\begin{proof}
We may upper bound $|a_n^* - \hat{a}_n|$ by the sum of two terms.
\begin{align*}
| a_n^* - \hat{a}_n | &\leq \left| \frac{1}{n_1} \sum_{Z_i = 1} [ \exp( \hat{\theta}_{1,n}^{\top} \mathbf{x}_i) - \exp( \theta_{1,n}^{* \top} \mathbf{x}_i)] \right|\\
&+ \left| \frac{1}{n_1} \sum_{Z_i = 1} [ \exp(\theta_{1,n}^{* \top} \mathbf{x}_i) - y_{1i}] - \frac{1}{n} \sum_{i = 1}^n [ \exp(\theta_{1,n}^{* \top} \mathbf{x}_i) - y_{1i}] \right| 
\end{align*}
The first of these terms tends to zero by Jensen's inequality and Lemma \ref{exponentiated_stability}.  The second term tends to zero by the completely randomized law of large numbers (Lemma \ref{wlln}).  The conditions needed to use the completely randomized law of large numbers are satisfied because $\tfrac{1}{n} \sum_{i = 1}^n [ \exp(\theta_{1,n}^{* \top} \mathbf{x}_i) - y_{1i}]^2 \leq \tfrac{2}{n} \sum_{i = 1}^n (\exp(\theta_{1,n}^{* \top} \mathbf{x}_i))^2 + \tfrac{2}{n} \sum_{i = 1}^n y_{1i}^2$.  Since $\theta_{1,n}^*$ and $|| \mathbf{x}_i ||$ are bounded (see the proof of Lemma \ref{log_ols_consistency}), the first term in this sum is $\mathcal{O}(1)$.  The second term is bounded by the assumption of Theorem \ref{log_transform}.
\end{proof}

\begin{lemma}
\label{second_stage_ols_consistency}
\textup{\textbf{(Consistency of second-stage OLS coefficients)}}\\
Define the following quantities:
\begin{align*}
\hat{\beta}_n \equiv (\hat{\beta}_{0,n}, \hat{\beta}_{1,n}) &= \argmin_{(\beta_0, \beta_1)} \sum_{Z_i = 1} (y_{1i} - [\beta_0 + \beta_1 \exp( \hat{\theta}_{1,n}^{\top} \mathbf{x}_i)])^2\\
\beta_n^* \equiv (\beta_{0,n}^*, \beta_{1,n}^*) &= \argmin_{(\beta_0, \beta_1)} \sum_{i = 1}^n (y_{1i} - [ \beta_0 + \beta_1 \exp( \theta_{1,n}^{* \top} \mathbf{x}_i)])^2
\end{align*}
Then $|| \hat{\beta}_n - \beta_n^* || \xrightarrow{p} 0$.
\end{lemma}

\begin{proof}
We start by proving that $| \hat{\beta}_{1,n} - \beta_{1,n}^* | \xrightarrow{p} 0$.  Both can be written explicitly:
\begin{align}
    \hat{\beta}_{1,n} &= \frac{\tfrac{1}{n_1} \sum_{Z_i = 1} y_{1i} \exp( \hat{\theta}_{1,n}^{\top} \mathbf{x}_i)}{\tfrac{1}{n_1} \sum_{Z_i = 1} [\exp( \hat{\theta}_{1,n}^{\top} \mathbf{x}_i)]^2} \label{beta1_hat}\\
    \beta_{1,n}^* &= \frac{\tfrac{1}{n} \sum_{i = 1}^n y_{1i} \exp(\theta_{1,n}^{* \top} \mathbf{x}_i)}{\tfrac{1}{n} \sum_{i = 1}^n [\exp(\theta_{1,n}^{* \top}  \mathbf{x}_i)]^2} \label{beta1_star}
\end{align}
The difference between the numerator of (\ref{beta1_hat}) and the numerator of (\ref{beta1_star}) is converging to zero.  To prove this, we first add and subtract a term, then apply the triangle inequality.
\begin{align*}
 \left| \frac{1}{n_1} \sum_{Z_i = 1} y_{1i} e^{\hat{\theta}_{1,n}^{\top} \mathbf{x}_i} - \frac{1}{n} \sum_{i = 1}^n y_{1i} e^{\theta_{1,n}^{* \top} \mathbf{x}_i} \right| &\leq \underbrace{\left| \frac{1}{n_1} \sum_{Z_i = 1} y_{1i} \left(  e^{\hat{\theta}_{1,n}^{\top} \mathbf{x}_i} - e^{\theta_{1,n}^{* \top} \mathbf{x}_i} \right) \right|}_{\text{(a)}}\\
 &+ \underbrace{\left| \frac{1}{n_1} \sum_{Z_i = 1} y_{1i} e^{\theta_{1,n}^{* \top} \mathbf{x}_i} - \frac{1}{n} \sum_{i = 1}^n y_{1i} e^{\theta_{1,n}^{* \top} \mathbf{x}_i} \right|}_{\text{(b)}}
\end{align*}
The term (a) is less than $(\tfrac{1}{n_1} \sum_{i = 1}^n y_{1i}^2)^{1/2} ( \tfrac{1}{n_1} \sum_{i = 1}^n [ \exp( \hat{\theta}_{1,n}^{\top} \mathbf{x}_i) - \exp( \theta_{1,n}^{* \top} \mathbf{x}_i)]^2)^{1/2}$.  Since $\tfrac{1}{n_1} \sum_{i = 1}^n y_{1i}^2$ is bounded by the assumptions of Theorem \ref{log_transform} and $\tfrac{1}{n_1} \sum_{i = 1}^n [ \exp( \hat{\theta}_{1,n}^{\top} \mathbf{x}_i) - \exp(\theta_{1,n}^{* \top} \mathbf{x}_i)]^2 \rightarrow_p 0$ by Lemma \ref{exponentiated_stability}, the term (a) is vanishing in probability.  The term (b) is vanishing by the completely randomized law of large numbers (Lemma \ref{wlln}).  The assumptions needed to use the completely randomized law of large numbers are satisfied because $\exp( \theta_{1,n}^{* \top} \mathbf{x}_i)$ is bounded (see the proof of Lemma \ref{log_ols_consistency}) and so is $\tfrac{1}{n} \sum_{i = 1}^n y_{1i}^2$.

By exactly the same argument, the difference between the denominators of (\ref{beta1_hat}) and (\ref{beta1_star}) can also be seen to be vanishing.  This is enough to prove $| \hat{\beta}_{1,n} - \beta_{1,n}^*| \xrightarrow{p} 0$, since the denominator in (\ref{beta1_star}) is bounded away from zero.
\begin{align*}
\frac{1}{n} \sum_{i = 1}^n [ \exp(\theta_{1,n}^{* \top} \mathbf{x}_i)]^2 &\geq \exp \left( \frac{1}{n} \sum_{i = 1}^n \theta_{1,n}^{* \top} \mathbf{x}_i \right) = \exp \left( \frac{1}{n} \sum_{i = 1}^n y_{1i} \right) \geq \exp \left( - \left[ \frac{1}{n} \sum_{i = 1}^n (\log y_{1i})^2 \right]^{1/2} \right)
\end{align*}
Since $\tfrac{1}{n} \sum_{i = 1}^n (\log y_{1i})^2$ is bounded above, the lower bound is bounded below.

Next, we prove that $| \hat{\beta}_{0,n} - \beta_{0,n}^* | \xrightarrow{p} 0$.  Again, both can be written explicitly:
\begin{align}
\hat{\beta}_{0,n} &= \frac{1}{n_1} \sum_{Z_i = 1} [ y_{1i} - \hat{\beta}_{1,n} \exp( \hat{\theta}_{1,n}^{\top} \mathbf{x}_i)] \label{beta0_hat}\\
\beta_{0,n}^* &= \frac{1}{n} \sum_{i = 1}^n [ y_{1i} - \beta_{1,n}^* \exp(\theta_{1,n}^{* \top} \mathbf{x}_i)] \label{beta0_star}
\end{align}
The difference between these two is bounded by the sum of of three terms:
\begin{align*}
| \hat{\beta}_{0,n} - \beta_{0,n}^*| &\leq
\underbrace{\left| \frac{1}{n_1} \sum_{Z_i = 1} y_{1i} - \frac{1}{n} \sum_{i = 1}^n y_{1i} \right|}_{\text{(i)}}\\
&+ \underbrace{\left| \frac{1}{n_1} \sum_{Z_i = 1} [\hat{\beta}_{1,n}
 \exp( \hat{\theta}_{1,n}^{\top} \mathbf{x}_i) - \beta_{1,n}^* \exp(\theta_{1,n}^{* \top} \mathbf{x}_i) \right|}_{\text{(ii)}}\\
 &+ \underbrace{\left| \frac{1}{n_1} \sum_{Z_i = 1} \beta_{1,n}^* \exp(\theta_{1,n}^{* \top} \mathbf{x}_i) - \frac{1}{n} \sum_{i = 1}^n \beta_{1,n}^* \exp(\theta_{1,n}^{* \top} \mathbf{x}_i) \right|}_{\text{(iii)}}
 \end{align*}
 The first and third terms are $o_p(1)$ by the completely randomized law of large numbers (to verify the assumptions on (iii), use the fact that the denominator of $\beta_{1,n}^*$ is bounded below and also $\exp(\theta_{1,n}^{* \top} \mathbf{x}_i)$ is bounded).  To kill the second term, add and subtract $\hat{\beta}_{1,n} \exp(\theta_{1,n}^{* \top} \mathbf{x}_i)$ and then use the Cauchy-Schwarz inequality along with Lemma \ref{exponentiated_stability} and the conclusion $| \hat{\beta}_{1,n} - \beta_{1,n}^* | \xrightarrow{p} 0$.
\end{proof}

\begin{lemma}
\label{smoothness_of_transformed_outcome}
\textup{\textbf{(Smoothness of debiased and post-OLS models)}}\\
Let $\phi = (\beta_0, \beta_1, \theta)$.  Define $\mu_{\phi}(\mathbf{x}) = \beta_0 + \beta_1 \exp(\theta^{\top} \mathbf{x})$.  Then for all $\phi, \phi'$ with $|| \phi ||_{\infty}, || \phi' ||_{\infty} \leq B$, we have:
\begin{align*}
|| \mu_{\phi} - \mu_{\phi'} ||_n \leq M || \phi - \phi' ||
\end{align*}
for some $M$ not depending on $n$.
\end{lemma}

\begin{proof}
First, we use the triangle inequality to split up $|| \mu_{\phi} - \mu_{\phi'} ||_n$ into more easily manageable terms.
\begin{align*}
|| \mu_{\phi} - \mu_{\phi'} ||_n \leq | \beta_0 - \beta_0'| + || \mu_{0, \beta_1, \theta} - \mu_{0, \beta_1', \theta} ||_n + || \mu_{0, \beta_1', \theta} - \mu_{0, \beta_1', \theta'} ||_n
\end{align*}
To analyze the second term, let $L < \infty$ be a uniform bound on $|| \mathbf{x}_i ||_1$.
\begin{align*}
|| \mu_{0, \beta_1, \theta} - \mu_{0, \beta_1', \theta} ||_n &= \left( \frac{1}{n} \sum_{i = 1}^n (\beta_1 - \beta_1')^2 \exp({\theta^{\top} \mathbf{x}_i})^2 \right)^{1/2} \leq \exp( 2 L B) | \beta_1 - \beta_1'|
\end{align*}
To analyze the third term, we use the local Lipschitz property of the exponential function.
\begin{align*}
|| \mu_{0, \beta_1', \theta} - \mu_{0, \beta_1', \theta'} ||_n &= \left( \frac{1}{n} \sum_{i = 1}^n (\beta_1')^2 [ \exp( \theta^{\top} \mathbf{x}_i) - \exp( \theta^{' \top} \mathbf{x}_i)]^2 \right)^{1/2}\\
&\leq B (e^{LB}) | \theta^{\top} \mathbf{x}_i - \theta^{' \top} \mathbf{x}_i |\\
&\leq LB e^{LB}  \cdot || \theta - \theta' ||
\end{align*}
Putting things together gives $|| \mu_{\phi} - \mu_{\phi'} ||_n \leq | \beta_0 - \beta_0' | + M_1 | \beta_1 - \beta_1' | + M_2 || \theta - \theta' ||$ for some $M_1, M_2$.  Then if we set $M = 2 \max \{ 1, M_1, M_2 \}$ we have:
\begin{align*}
|| \mu_{\theta} - \mu_{\theta'} ||_n &\leq \max \{ 1, M_1, M_2 \} \left( | \beta_0 - \beta_0'| + | \beta_1 - \beta_1'| + \left( \sum_{j = 1}^d (\theta_j - \theta_j')^2 \right)^{1/2}\right)\\
&\leq 2 \max \{ 1, M_1, M_2 \}  \left( | \beta_0 - \beta_0'|^2 + | \beta_1 - \beta_1'|^2 + \sum_{j = 1}^d (\theta_j - \theta_j')^2 \right)^{1/2}\\
&= M || \phi - \phi' ||
\end{align*}
where we used the inequality $\sqrt{a} + \sqrt{b} \leq \sqrt{2} \sqrt{a + b}$ twice.
\end{proof}

\begin{proposition}
\textup{\textbf{(Proof for the debiased estimator)}}\\
The sequence of debiased estimators $\{ \hat{\mu}_{1,n}^{\mathsf{db}} \}$ is stable and has typically simple realizations.
\end{proposition}

\begin{proof}
Define the ``population" debiased regression function $\mu_{1,n}^{\mathsf{db}*}$ by:
\begin{align*}
    \mu_{1,n}^{\mathsf{db}*}(\mathbf{x}) := \exp( \theta_{1,n}^{* \top} \mathbf{x}) - \underbrace{\frac{1}{n} \sum_{i = 1}^n [\exp(\theta_{1,n}^{* \top} 
    \mathbf{x}_i) - y_{1i}]}_{= a_n^*}.
\end{align*}
where $\theta_{1,n}^*$ is defined in Lemma \ref{log_ols_consistency}.  The following calculation shows that $|| \hat{\mu}_{1,n}^{\mathsf{db}} - \mu_{1,n}^{\mathsf{db} *} ||_n$ is vanishing in probability, i.e. $\{ \hat{\mu}^{\mathsf{db}}_{1,n} \}_{n \geq 1}$ is a stable sequence.
\begin{align*}
|| \hat{\mu}_{1,n}^{\mathsf{db}} - \mu_{1,n}^{\mathsf{db*}} ||_n &\leq || \exp( \langle \hat{\theta}_{1,n}, \cdot \rangle) - \exp( \langle \theta_{1,n}^*, \cdot \rangle) ||_n + || \hat{a}_n - a_n^* ||_n\\
&= \underbrace{|| \exp( \hat{\theta}_{1,n}, \cdot \rangle) - \exp( \theta_{1,n}^*, \cdot \rangle) ||_n}_{=o_p(1) \text{ by Lemma \ref{exponentiated_stability}}} + \underbrace{| a_n - a_n^* |}_{= o_p(1) \text{ by Lemma \ref{debiased_intercept_consistency}}}
\end{align*}
To check that $\{ \hat{\mu}^{\mathsf{db}}_{1,n} \}_{n \geq 1}$ has typically simple realizations, define $\mathcal{F}_n = \{ \mu_{-a, 1, \theta}(\mathbf{x}) = \exp(\theta^{\top} \mathbf{x}) - a \, : \, |a - a_n^*| \leq 1, || \theta - \theta_{1,n}^* || \leq 1 \}$.  Then $\P_{n_1,n}( \hat{\mu}_{1,n}^{\mathsf{db}} \in \mathcal{F} ) \geq 1 - \P_{n_1,n}(|\hat{a}_n - a_n^*| > 1) - \P_{n_1,n}( || \hat{\theta}_{1,n} - \theta_{1,n}^* || > 1)$.  Since these two probabilities are vanishing by Lemma \ref{exponentiated_stability} and Lemma \ref{debiased_intercept_consistency}, we have $\P_{n_1,n}( \hat{\mu}_{1,n}^{\mathsf{db}} \in \mathcal{F}) \rightarrow 1.$  Since $a_n^*$ and $\theta_{1,n}^*$ are bounded, Lemma \ref{smoothness_of_transformed_outcome} and the argument from Example \ref{ols_is_simple} show that $\mathcal{F}_n$ satisfies the entropy integral condition.
\end{proof}

\begin{proposition}
\textup{\textbf{(Proof for the second-stage OLS estimator)}}\\
The sequence of second-stage OLS estimators $\{ \hat{\mu}_{1,n}^{\mathsf{ols2}} \}_{n \geq 1}$ is stable and has typically simple realizations.
\end{proposition}

\begin{proof}
Define the ``population" second-stage OLS regression function $\mu_{1,n}^{\mathsf{ols2}*}$ by:
\begin{align*}
    \mu_{1,n}^{\mathsf{ols2}*}(\mathbf{x}) := \beta_{0,n}^* + \beta_{1,n}^* \exp( \theta_{1,n}^{* \top} \mathbf{x})
\end{align*}
where $\beta_{0,n}^*$ and $\beta_{1,n}^*$ are defined in Lemma \ref{second_stage_ols_consistency}.  Then $|| \hat{\mu}_{1,n}^{\mathsf{ols2}} - \mu_{1,n}^{\mathsf{ols2}*} ||_n \xrightarrow{p} 0$ by first splitting into three terms,
\begin{align*}
|| \hat{\mu}_{1,n}^{\mathsf{ols2}} - \mu_{1,n}^{\mathsf{ols2}*} ||_n &\leq | \hat{\beta}_{0,n} - \beta_{0,n}^* |\\
&+ || \hat{\beta}_{1,n} \exp( \langle \hat{\theta}_{1,n}, \cdot \rangle) - \beta_{1,n}^* \exp( \langle \hat{\theta}_{1,n}, \cdot \rangle) ||_n\\
&+ || \beta_{1,n}^* \exp( \langle \hat{\theta}_{1,n}, \cdot \rangle) - \beta_{1,n}^* \exp( \langle \theta_{1,n}^*, \cdot \rangle) ||_n
\end{align*}
then applying Lemma \ref{second_stage_ols_consistency} and Lemma \ref{exponentiated_stability} to kill off each term in the upper bound.  Lemma \ref{exponentiated_stability} ca be applied on the final summand because $\beta_{1,n}^*$ is bounded.  This can be seen by examining the ratio formula for $\beta_{1,n}^*$ given in the proof of Lemma \ref{log_ols_consistency}.  The proof of that Lemma gives a uniform lower bound of the denominator in that ratio, and the boundedness of the numerator follows because $\tfrac{1}{n} \sum_{i = 1}^n y_{1i}^2$, $|| \mathbf{x}_i ||$, and $\theta_{1,n}^*$ are all bounded.

The ``typically simple realizations" condition is checked in the same way as in the case of the debiased estimator, using Lemma \ref{smoothness_of_transformed_outcome} and the argument from Example \ref{ols_is_stable}. 
\end{proof}

\subsubsection{Proof of Theorem \ref{isotonic_regression}}

\begin{proof}
Prediction unbiasedness is established in Barlow \& Brunk \citep{isotonic}.  First, we will prove that $\{ \hat{\mu}_{1,n} \}_{n \geq 1}$ has ``typically simple realizations."  Without loss of generality, suppose that $[a, b] = [0, 1]$.  By Lemma 9.11 in \cite{kosorok2008introduction}, the class $\mathsf{M}$ of monotone functions taking $\R$ into $[0, 1]$ satisfies the following metric entropy bound
\begin{align}
\label{monotone_metric_entropy}
    \log \mathsf{N} ( \mathsf{M}, || \cdot ||_n, s) \leq \frac{K}{s}
\end{align}
for all $s \in (0, 1)$ and any $n \geq 1$.  In the above display, $K < \infty$ is a universal constant.  Let $\mathcal{F}_n = \mathsf{M}$.  Using (\ref{monotone_metric_entropy}), we may write:
\begin{align*}
\int_0^{1} \sup_{n \geq 1} \sqrt{\log \mathsf{N}( \mathcal{F}_n, || \cdot ||_n, s)} \mathsf{d} s &\leq \int_0^1 \sqrt{K/s} \, \d s < \infty
\end{align*}
Since $\P( \hat{\mu}_{1,n} \in \mathcal{F}_n) = 1$ for all $n$, this shows that $\{ \hat{\mu}_{1,n} \}$ satisfies the ``typically simple realizations" condition.

Next, we need to show that $\{ \hat{\mu}_{1,n} \}$ is stable.  Let $\mu_{1,n}^*$ be the solution to the ``population" isotonic regression problem.  For any $\mu \in \mathsf{M}$, define $\ell_{\mu}( y, x)$ by $\ell_{\mu} (y, x) = (y - \mu(x))^2$.  The distance\footnote{Technically, the distance $|| \cdot ||_n$ on the functions $\ell_{\mu}$ is not the same as the distance $|| \cdot ||_n$ on the functions $\mu$, because the input is one dimension higher.  For notational cleanliness, we avoided introducing a new norm.} between $\ell_{\mu}$ and $\ell_{\nu}$ can be controlled by $|| \mu - \nu ||_n$.  This uses the fact that $(y_{1i} - \mu(x_i))^2$ is between 0 and 1 for any $i$.
\begin{align*}
|| \ell_{\mu} - \ell_{\nu} ||_n^2 &= \frac{1}{n} \sum_{i = 1}^n [(y_{1i} - \mu(x_i))^2 - (y_{1i} - \nu(x_i))^2]^2\\
&\leq \frac{1}{n} \sum_{i = 1}^n | (y_{1i} - \mu(x_i))^2 - (y_{1i} - \nu(x_i))^2|^2\\
&= \frac{1}{n} \sum_{i = 1}^n |y_{1i} - \mu(x_i) + y_{1i} - \nu(x_i)|^2 \cdot |\mu(x_i) - \nu(x_i)|^2\\
&\leq \frac{4}{n} \sum_{i = 1}^n | \mu(x_i) - \nu(x_i)|^2\\
&= 4 || \mu - \nu ||_n^2
\end{align*}
Hence, the $s$-covering number of the set $\mathsf{L} = \{ \ell_{\mu} : \mu \in \mathsf{M} \}$ in the $|| \cdot ||_n$-norm is upper bounded by the $s/2$-covering number of the set $\mathsf{M}$.  Thus, if we define $\mathcal{R}(\mu) = \tfrac{1}{n} \sum_{i = 1}^n (y_{1i} - \mu(x_i))^2$ and $\hat{\mathcal{R}}(\mu) = \tfrac{1}{n_1} \sum_{Z_i = 1} (y_{1i} - \mu(x_i))^2$, then Proposition~\ref{crd_maximal_inequality} implies:
\begin{align*}
\E \left[ \sup_{\mu \in \mathsf{M}} | [ \hat{\mathcal{R}}(\mu) - \hat{\mathcal{R}}(\mu_{1,n}^*)] - [\mathcal{R}(\mu) - \mathcal{R}(\mu_{1,n}^*)]| \right] &= \frac{1}{\sqrt{n}} \E \left[ \sup_{\ell \in \mathsf{L}}  | \mathbb{G}_n(\ell) - \mathbb{G}_n( \ell_{\mu_{1,n}^*})| \right]\\
&\leq \frac{C / p_{\min}}{\sqrt{n}} \int_0^{1} \sqrt{\log \mathsf{N}( \mathsf{M}, || \cdot ||_n, s/2)} \, \mathsf{d} s\\
&\leq \frac{C / p_{\min}}{\sqrt{n}} \int_0^{1} \sqrt{2K / s} \, \mathsf{d} s\\
&\equiv C' / \sqrt{n}
\end{align*}
Hence, for any $\epsilon > 0$, the quantity in the expectation on the LHS of the above display is less than $\epsilon$ with probability tending to one.  On that event, we have:
\begin{align*}
0 \leq \mathcal{R}(\hat{\mu}_{1,n}) - \mathcal{R}(\mu_{1,n}^*) &\leq \hat{\mathcal{R}}( \hat{\mu}_{1,n}) - \mathcal{\hat{R}}( \mu_{1,n}^*) + \epsilon \leq \epsilon
\end{align*}
Since $\epsilon$ is arbitrary, this proves $|\mathcal{R}( \hat{\mu}_{1,n}) - \mathcal{R}(\mu_{1,n}^*)|$ tends to zero in probability.  Finally, to prove that $|| \hat{\mu}_{1,n} - \mu_{1,n}^* ||_n$ is tending to zero, we appeal to the convexity of the set $\mathsf{M}_n := \{ \mu \in \R^n : \mu_1 \leq \mu_2 \leq \cdots \leq \mu_n \}$.  With a slight abuse of notation, let $\mu_{1,n}^* = (\mu_{1,n}^*(x_{(1)}), \cdots, \mu_{1,n}^*(x_{(n)})) \in \mathsf{M}_n$ and let $\hat{\mu}_{1,n} \in \mathsf{M}_n$ be defined similarly.  Then $\mu_{1,n}^*$ is the projection of $\mathbf{y}_1$ onto $\mathsf{M}_n$, so $\langle \mu_{1,n}^* - \mathbf{y}_1, \hat{\mu}_{1,n} - \mu_{1,n}^* \rangle \geq 0$.  This allows us to write:
\begin{align*}
|| \hat{\mu}_{1,n} - \mathbf{y}_1 ||^2 &= || (\hat{\mu}_{1,n} - \mu_{1,n}^*) + (\mu_{1,n}^* - \mathbf{y}_1) ||^2\\
&= || \hat{\mu}_{1,n} - \mu_{1,n}^* ||^2 + || \mu_{1,n}^* - \mathbf{y}_1 ||^2 + 2 \langle \hat{\mu}_{1,n} - \mu_{1,n}^*, \mu_{1,n}^* - \mathbf{y}_1 \rangle\\
&\geq || \hat{\mu}_{1,n} - \mu_{1,n}^* ||^2 + || \mu_{1,n}^* - \mathbf{y}_1 ||^2\\
|| \hat{\mu}_{1,n} - \mu_{1,n}^* ||^2 &\leq || \hat{\mu}_{1,n} - \mathbf{y}_1 ||^2 - || \mu_{1,n}^* - \mathbf{y}_1 ||^2\\
&= \mathcal{R}(\hat{\mu}_{1,n}) - \mathcal{R}(\mu_{1,n}^*)\\
&\xrightarrow{p} 0
\end{align*}
\end{proof}

\subsubsection{Proof of Lemma \ref{growth_lemma}}

\begin{proof}
For all $\theta \in \mathbb{B}_r(\theta_n^*)$, the mean-value theorem allows us to write:
\begin{align*}
    \L_n(\theta) &= \L_n(\theta_n^*) + \nabla \L_n(\theta_n^*) (\theta - \theta_n^*) + \frac{1}{2} (\theta - \theta_n^*)^{\top} \nabla^2 \L_n( \bar{\theta})(\theta - \theta_n^*)
\end{align*}
In the above display, $\bar{\theta} = \lambda \theta + (1 - \lambda) \theta_n^*$ for some $\lambda \in [0, 1]$.  Using the fact that $\nabla \L_n(\theta_n^*) = 0$ and $\Lambda_{\min}(\nabla^2 \L_n(\bar{\theta})) \geq \lambda_{\min}$, we may conclude $\L_n(\theta) - \L_n(\theta_n^*) \geq \tfrac{1}{2} \lambda_{\min} || \theta - \theta_n^* ||^2$.  Therefore, Assumption \ref{growth} is satisfied with the following choice of the function $f$:
\begin{align*}
    f(t) = 
    \left\{
    \begin{array}{ll}
    \tfrac{1}{2} \lambda_{\min} t^2 &t < r\\
    (\tfrac{1}{2} \lambda_{\min} r) t &t \geq r
    \end{array} \right.
\end{align*}
The extension to $t \geq r$ was by convexity.
\end{proof}

\subsubsection{Proof of Theorem \ref{general_parametric_models}}

\begin{proof}
First, we prove a useful uniform convergence result.  Let $\ell_{\theta} : \R^{d + 1} \rightarrow \R$ be defined by $\ell_{\theta}(\mathbf{x}, y) = \ell(\theta, \mathbf{x}, y)$.  For any radius $\kappa \geq 0$, define $\mathcal{F}_n(\kappa) := \{ \ell_{\theta} \, : \, || \theta - \theta_n^* || \leq \kappa \}$.  If we set $\mathbb{G}_n(\theta) = \sqrt{n}[ \hat{\L}_n(\theta) - \L_n(\theta)]$, then Proposition \ref{crd_maximal_inequality} gives an average-case bound for the supremum of $\mathbb{G}_n(\theta)$ near $\theta_n^*$.
\begin{align}
\label{localized_maximal_inequality}
\E \left[ \sup_{|| \theta - \theta_n^* || \leq \kappa} | \mathbb{G}_n(\theta) - \mathbb{G}_n(\theta_n^*)| \right] &\leq (C / p_{\min}) \int_0^{\infty} \sqrt{\log \mathsf{N}(\mathcal{F}_n(\kappa), || \cdot ||_n', s)} \, \d s 
\end{align}
In the above display, $C$ is a universal constant and $|| f - g ||_n' := (\tfrac{1}{n} \sum_{i = 1}^n [f(\mathbf{x}_{i,n}, y_{1i,n}) - g(\mathbf{x}_{i,n}, y_{1i,n})]^2 )^{1/2}$.  If $\kappa$ is smaller than the radius $r$ from assumption \ref{smoothness}, then the $s$-covering number of $(\mathcal{F}_n(\kappa), || \cdot ||_n)$ can be bounded by the $(s/L)$-covering number of the Euclidean ball $\mathbb{B}_{\kappa}(\theta_n^*)$;  this is because any $(s/L)$-covering $\{ \theta_i \}_{i = 1}^N$ of $\mathbb{B}_{\kappa}(\theta_n^*)$ gives an $s$-covering $\{ \ell_{\theta_i} \}_{i = 1}^N$.  It can be shown that the $(s/L)$-covering number of a Euclidean ball of radius $\kappa$ is bounded by $\kappa^d (1 + 2L/s)^d$ (see Chapter 4 of \cite{vershynin_2018}).  Therefore, the upper bound can be further bounded as follows:
\begin{align*}
\int_0^{\infty} \sqrt{\log \mathsf{N}( \mathcal{F}_n(\kappa), || \cdot ||_n', s)} \, \d s &\leq  \int_0^{\infty} \sqrt{\log \mathsf{N}( \mathbb{B}_{\kappa}(\theta_n^*), \ell_2, s/L)} \, \d s\\
&\leq \int_0^{\kappa} \sqrt{d \log \kappa + d \log(1 + 2L/s)} \, \d s\\
&\leq \kappa \sqrt{d \log \kappa} + \sqrt{8 d L \kappa}
\end{align*}
In particular, for any $\kappa$, the upper bound in (\ref{localized_maximal_inequality}) is bounded by a constant $M(\kappa) := (C / p_{\min}) (\kappa \sqrt{d \log \kappa} + \sqrt{8 d L \kappa})$ not depending on $n$.  The conclusion $\sup_{|| \theta - \theta_n^* || \leq \kappa} | [\hat{\L}_n(\theta) - \hat{\L}_n(\theta_n^*)] - [\L_n(\theta) - \L_n(\theta_n^*)]| \rightarrow_p 0$ follows from Markov's inequality.

Now, we are ready to prove $|| \hat{\theta}_n - \theta_n^* || \rightarrow_p 0$.  Let $\epsilon \in (0, r)$ be arbitrary.  For all large $n$, $\L_n(\theta) - \L_n(\theta_n^*) \geq f(\epsilon) > 0$ whenever $|| \theta - \theta_n^* || = \epsilon$.  By the preceding uniform convergence result, that means $\hat{\L}_n(\theta) - \hat{\L}_n(\theta_n^*) \geq f(\epsilon)/2$ simultaneously for all $\theta$ satisfying $|| \theta - \theta_n^* || = \epsilon$ with probability going to one.  By convexity, that event implies $\hat{\theta}_n \in \mathbb{B}_{\epsilon}(\theta_n^*)$.  Since $\epsilon$ is arbitrary, that proves $|| \hat{\theta}_n - \theta_n^* || \rightarrow_p 0$.

If the map $\theta \mapsto \mu_{\theta}$ is also smooth in the sense $|| \mu_{\theta} - \mu_{\phi} ||_n \leq M || \theta - \phi ||$ for all $\theta, \phi \in \mathbb{B}_r(\theta_n^*)$, then stability follows immediately by setting $\mu_n^* = \mu_{\theta_n^*}$.  The ``typically simple realizations" condition can be verified by setting $\mathcal{F}_n = \{ \mu_{\theta} \, : \, || \theta - \theta_n^* || \leq 1 \}$.  This follows by once again using the covering numbers of the ball $\mathbb{B}_1(\theta_n^*)$ to upper bound the covering number of $\mathcal{F}_n$ -- the argument is spelled out in Example \ref{ols_is_simple}.
\end{proof}

\end{document}